\documentclass[11pt]{amsart}
\usepackage{macro}

\title{Height one specializations of Selmer groups}
\author{Bharathwaj Palvannan}
\address{Department of Mathematics  \\ University of Pennsylvania \\  209 South 33rd Street, 4N53 DRL, \\ Philadelphia 19104-6395, USA}
\email{pbharath@math.upenn.edu}

\begin{document}

\begin{abstract}
We provide applications to studying the behavior of Selmer groups under specialization. We consider Selmer groups associated to four dimensional Galois representations coming from (i) the tensor product of two cuspidal Hida families $F$ and $G$, (ii) its cyclotomic deformation,  (iii) the tensor product of a cusp form $f$ and the Hida family $G$, where $f$ is a classical specialization of $F$ with weight $k \geq 2$. We prove control theorems to  relate (a) the Selmer group associated to the tensor product of Hida families $F$ and $G$ to the Selmer group associated to its cyclotomic deformation and (b) the Selmer group associated to the tensor product of $f$ and $G$ to the Selmer group associated to the tensor product of $F$ and $G$. On the analytic side of the main conjectures, Hida has constructed one variable, two variable and three variable Rankin-Selberg $p$-adic $L$-functions. Our specialization results enable us to verify that Hida's results relating (a) the two variable $p$-adic $L$-function to the three variable  $p$-adic $L$-function and (b) the one variable  $p$-adic $L$-function to the two variable $p$-adic $L$-function and our control theorems for Selmer groups are completely consistent with the main conjectures.
\end{abstract}

\maketitle

Studying the behavior of Selmer groups under specialization has been quite fruitful in the context of the main conjectures in Iwasawa theory. We cite three examples where the topic of specializations of Selmer groups has come up in the past; it has come up in a work of Greenberg \cite{greenberg26iwasawa} to prove a result involving classical Iwasawa modules corresponding to a factorization formula proved by Gross involving Katz's $2$-variable $p$-adic $L$-function, in the work of Greenberg-Vatsal \cite{greenberg2000iwasawa} while studying an elliptic curve with an isogeny of degree $p$ and in an earlier work  of ours \cite{bharathwaj2016algebraic} to prove a result involving Selmer groups corresponding to a factorization formula proved by Dasgupta involving Hida's Rankin-Selberg $3$-variable $p$-adic $L$-function. \\

The main results of this paper will deal with Selmer groups to which the Rankin-Selberg $p$-adic $L$-functions, constructed by Hida in \cite{hida1988p}, are associated. The results of Lei-Loeffler-Zerbes \cite{lei2012euler} and Kings-Loeffler-Zerbes \cite{kings2015rankin} suggest that the proofs of the main conjectures associated to these Rankin-Selberg $p$-adic $L$-functions are imminent;  the results of this paper provide evidence to these conjectures.  The results of this paper contrast well with the three examples cited earlier since in this paper we deal with ``critical'' specializations (the three examples cited earlier dealt with ``non-critical'' specializations). After the completion of this project, the author learnt about the paper \cite{hara2015cyclotomic} where the behavior of Selmer groups under certain critical specializations has also been studied. Hara and Ochiai (see Theorem D in \cite{hara2015cyclotomic}) deduce the cyclotomic Iwasawa main conjecture over a CM field $E$ (formulated over a regular local ring with Krull dimension $2$) for Hilbert cuspforms with complex multiplication from the main conjecture formulated (over a regular local ring with Krull dimension greater than  or equal to $\frac{[E:\Q]}{2}+2$) by Hida and Tilouine in \cite{hida1994anticyclotomic}; the Hida-Tilouine main conjecture has been proved by Ming-Lun Hsieh \cite{hsieh2014eisenstein} under certain hypotheses. To place our work involving specializations of Selmer groups in the context of main conjectures, this result of \cite{hara2015cyclotomic} might serve as a useful template for the reader to keep in mind. \\

To prove our results, we will use the techniques developed in an earlier work of ours \cite{bharathwaj2016algebraic} (which is briefly summarized in Section \ref{towards-the-proof}). The novelty in the techniques introduced in \cite{bharathwaj2016algebraic} is that the local domains, appearing in Hida theory and over which the main conjectures  are formulated, are not always known to be regular or even a UFD. See Section \ref{examples} where we provide such examples. These examples are based on the circle of ideas first developed by Hida in \cite{hida1998global} and later explored in works of Cho \cite{cho1999deformation} and Cho-Vatsal \cite{cho2003deformations}.  In contrast, the main conjectures over CM fields are formulated over regular local rings.  One obstacle from the perspective of commutative algebra that thus needs to be overcome is that the kernel of the specialization map (which in our case is a prime ideal of height one) is not known to be principal. See Remark \ref{obstacle} \\

Let us introduce some notations. Let $p \geq 5$ be a fixed prime number. Let $F=\sum_{n=1}^\infty a_n q^n~\in~R_{F}[[q]]$ and $G=\sum_{n=1}^\infty b_n q^n \in R_G[[q]]$ be two ordinary cuspidal Hida families. Here, we let $R_F$ and $R_G$ denote the integral closures of the irreducible components of the ordinary primitive cuspidal Hecke algebras through $F$ and $G$ respectively. The rings $\R_F$ and $\R_G$ turn out to be finite integral extensions of $\Z_p[[x_F]]$ and $\Z_p[[x_G]]$ respectively, where $x_F$ and $x_G$ are ``weight variables'' for $F$ and $G$ respectively. We assume the following hypotheses on $F$ and $G$:

\begin{enumerate}[style=sameline,leftmargin=2cm, style=sameline, align=left,label=\textsc{IRR}, ref=\textsc{IRR}]
\item\label{residual-irr} The residual representations associated to $F$ and $G$ are absolutely irreducible.
\end{enumerate}
\begin{enumerate}[style=sameline, leftmargin=2cm, style=sameline, align=left,label=\textsc{$p$-DIS-IN}, ref=\textsc{$p$-DIS-IN}]
\item \label{p-distinguished-inertia} The restrictions, to the inertia subgroup $I_p$ at $p$, of the residual representations associated to $F$ and $G$  have non-scalar semi-simplifications.
\end{enumerate}
Let $\Sigma$ be a finite set of primes in $\Q$ containing the primes $p$, $\infty$, a finite prime $l_0 \neq p$ and all the primes dividing the levels of $F$ and $G$. Let $\Sigma_0 = \Sigma \setminus \{p\}$. Let $O$ denote the ring of integers in a finite extension of $\Q_p$. Let $\Q_\Sigma$ denote the maximal extension of $\Q$ unramified outside $\Sigma$.  Let $G_\Sigma$ denote $\Gal{\Q_\Sigma}{\Q}$. We have Galois representations $\rho_F: G_\Sigma \rightarrow \Gl_2(\R_F)$ and $\rho_G: G_\Sigma \rightarrow \Gl_2(\R_G)$ associated to $F$ and $G$ respectively (see \cite{hida1986galois}). We let $L_F$ (and $L_G$) be the free $\R_F$-module (and $R_G$-module) of rank $2$ on which $G_\Sigma$ acts to let us obtain the Galois representation $\rho_F$ (and $\rho_G$ respectively).  Assume that the integral closures of $\Z_p$ in $\R_F$ and $\R_G$ are equal (extending scalars if necessary)\footnote{The ring $R_{F,G}$ turns out to be a domain. This is because the completed tensor product $R_{F,G}$ can be shown to be isomorphic to the tensor product $R_{F}[[x_G]] \otimes_{O[[x_G]]} R_G$. The fraction field of $R_G$ is a finite field extension of the fraction field of $O[[x_G]]$, while the fraction field of $R_{F}[[x_G]]$ is a purely transcendental extension of the fraction field of $O[[x_G]]$. Since the fraction fields of $R_{F}[[x_G]] $ and $R_G$ are ``\textit{linearly disjoint over the fraction field of $O[[x_G]]$}'', the tensor product $R_{F}[[x_G]] \otimes_{O[[x_G]]} R_G$ (and hence $R_{F,G}$) turns out to be a domain.} to $O$.

\subsection*{Specializing from 3-variables to 2-variables}
Let $\R_{F,G}$ denote the completed tensor product $\R_F \hotimes \R_G$. The completed tensor product is the co-product in the category of complete semi-local Noetherian $O$-algebras (where the morphisms are continuous $O$-algebra maps).  We have the natural inclusions $i_F:\R_F\hookrightarrow\R_{F,G}$ and $i_G:\R_{G} \hookrightarrow\R_{F,G}$.
One can construct a 4-dimensional Galois representation $\rho_{F,G}~:~G_\Sigma~\rightarrow~\Gl_4(\R_{F,G})$, that is given by the action of $G_\Sigma$ on the following free $\R_{F,G}$-module of rank $4$:
 \begin{align*} L_2 := \Hom_{\R_{F,G}} \bigg(L_F \otimes_{R_F} \R_{F,G},\ \ L_G \otimes_{R_G} \R_{F,G} \bigg).\end{align*}
Let $\rho_{F,G} \otimes \kappa^{-1} : G_\Sigma \rightarrow \Gl_4(\R_{F,G}[[\Gamma]])$ denote the Galois representation  given by the action of $G_\Sigma$ on the following free $R_{F,G}[[\Gamma]]$-module of rank $4$:

\begin{align*}L_3 := L_2 \otimes_{\R_{F,G}} \R_{F,G}[[\Gamma]] (\kappa^{-1}).\end{align*}
We let $\Q_\infty$ denote the cyclotomic $\Z_p$-extension of $\Q$ and we let $\Gamma$ denote the Galois group $\Gal{\Q_\infty}{\Q}$. Here, $\kappa : G_\Sigma \twoheadrightarrow \Gamma \hookrightarrow~\Z_p[[\Gamma]]^\times$ denotes the tautological character. To simplify notations, we will let $\rho_2$ denote $\rho_{F,G}$ and $\rho_3$ denote $\rho_{F,G} \otimes \kappa^{-1}$. Hida has constructed elements $\theta^{\Sigma_0}_2$ and $\theta^{\Sigma_0}_3$, in the fraction fields of $\R_{F,G}$ and $\R_{F,G}[[\Gamma]]$ respectively, as generators for non-primitive $p$-adic $L$-functions associated to $\rho_2$ and $\rho_3$ respectively. See section 7.4 and section 10.4 of \cite{hida1993elementary} for their constructions and the precise interpolation properties that they satisfy. The elements $\theta^{\Sigma_0}_2$ and $\theta^{\Sigma_0}_3$ generate a 2-variable $p$-adic $L$-function and 3-variable $p$-adic $L$-function respectively. For the 2-variable $p$-adic $L$-function, one can vary the weights of $F$ and $G$. For the 3-variable $p$-adic $L$-function, in addition to the weights of $F$ and $G$, one can also vary the cyclotomic variable. Notice that the subscripts in our notation for the Galois representations or the associated Galois lattices indicate the number of variables in the corresponding $p$-adic $L$-functions. \\

Let $\pi_{3,2} : \R_{F,G}[[\Gamma]] \rightarrow \R_{F,G}$ be the map defined by sending every element of $\Gamma$ to $1$. We have the Galois representation   $\pi_{3,2}\circ\rho_3:G_\Sigma\xrightarrow {\rho_3}\Gl_4(\R_{F,G}[[\Gamma]])\xrightarrow {\pi_{3,2}}\Gl_4(\R_{F,G})$, along with the following isomorphism of Galois representations:
\begin{align}
\pi_{3,2} \circ \rho_3 \cong \rho_2.
\end{align}
Hida has proved the following theorem relating $\theta^{\Sigma_0}_2$ and $\theta^{\Sigma_0}_3$. See Theorem 1 in Section 10.5~of~\cite{hida1993elementary} and Theorem 5.1d' in \cite{hida1988p}.

\begin{Theorem}[Hida] \label{analysis-3-2} $\pi_{3,2}\left(\theta^{\Sigma_0}_{3}\right) = \left( 1-\frac{i_F(a_p)}{i_G(b_p)} \right) \cdot \theta^{\Sigma_0}_{2}$.
\end{Theorem}

Recall that the ordinariness assumption on $F$ and $G$ ensures us that the coefficients $i_F(a_p)$ and $i_G(b_p)$ would be units in the ring $\R_{F,G}$.  On the algebraic side, we can construct Selmer groups $\Sel_{\rho_2}(\Q)$ and $\Sel_{\rho_3}(\Q)$ associated to the Galois representations $\rho_2$ and $\rho_3$ respectively. We prove the following control theorem relating $\Sel_{\rho_2}(\Q)$ and $\Sel_{\rho_3}(\Q)$:

\begin{Theorem} \label{algebra-3-2}
We have the following equality in the divisor group of $\R_{F,G}$:
\begin{align*}
 \Div \left( \Sel_{\rho_3}(\Q)^\vee \otimes_{R_{F,G}[[\Gamma]]} \R_{F,G} \right) &+ \Div \bigg(H^0(G_\Sigma, D_{\rho_3})^\vee [\ker(\pi_{3,2})]\bigg)  \\& =  \Div \left( \Sel_{\rho_2}(\Q)^\vee \right) + \Div \left( 1-\frac{i_F(a_p)}{i_G(b_p)} \right) .
\end{align*}
\end{Theorem}
Let us explain some of our notations. If $A$ is a profinite ring, we let $\hat{A}$ denote the Pontryagin dual of $A$. If $M$ is a continuous $A$-module, we let $M^\vee$ denote its Pontryagin dual. To a finitely generated torsion module over an integrally closed Noetherian local domain or a non-zero element of its fraction field, we shall associate an element of the divisor group following  \cite{greenberg1994iwasawa}. If $M$ is an $A$-module and $I$ is an ideal of $A$, the $A/I$-module
$M[I]$ is defined to be the set $\{ m \in M \ | \ i \cdot m =0, \ \forall i \in I\}$. The discrete modules $D_{\rho_3}$, $D_{\rho_2}$ and $D_{\rho_1}$ will be precisely defined in Section \ref{selmer}. We will assume throughout this paper that the $R_{F,G}[[\Gamma]]$-module $\Sel_{\rho_3}(\Q)^\vee$ and the $R_{F,G}$-module $\Sel_{\rho_2}(\Q)^\vee$ are torsion. \\

One can formulate a main conjecture in Iwasawa theory for $\rho_2$ and $\rho_3$ following \cite{greenberg1994iwasawa}. The main conjecture for $\rho_3$ predicts the  following equality of divisors in $R_{F,G}[[\Gamma]]$:
\begin{align} \label{MC3} \tag{MC-$\rho_3$}
 \Div \left(\theta_3^{\Sigma_0} \right)   \stackrel{?}{=}\Div \left(\Sel_{\rho_3}(\Q)^\vee\right) - \Div \left( H^0(G_\Sigma, D_{\rho_3})^\vee \right).
\end{align}
The main conjecture for $\rho_2$ predicts the  following equality of divisors in $R_{F,G}$:
\begin{align} \label{MC2} \tag{MC-$\rho_2$}
 \Div \left(\theta_2^{\Sigma_0} \right)   \stackrel{?}{=}\Div \left(\Sel_{\rho_2}(\Q)^\vee\right) - \Div \left( H^0(G_\Sigma, D_{\rho_2})^\vee \right).
\end{align}
We prove the following theorem to show that Theorem \ref{analysis-3-2} and Theorem \ref{algebra-3-2} are completely consistent with the main conjectures \ref{MC3} and \ref{MC2}:

\begin{Theorem} \label{specialization-3-2}
If the main conjecture \ref{MC3} holds, then the main conjecture \ref{MC2} also holds.
\end{Theorem}

\begin{Remark} \label{tor-preference}
As we will show in equation (\ref{tor-iso-3-2}), we have the following isomorphism of $R_{F,G}$-modules:
\begin{align}\label{tor-h0-iso}
\Tor_{1}^{R_{F,G}[[\Gamma]]} \left( R_{F,G}, \ H^0(G_{\Sigma},D_{\rho_3})^\vee \right) \cong H^0(G_\Sigma, D_{\rho_3})^\vee[\ker(\pi_{3,2})].
\end{align}
Regarding Theorem \ref{algebra-3-2}, it will be useful to keep this isomorphism in mind. Since we invoke results from homological algebra, we will need to use the Tor functor frequently.
\end{Remark}
\subsection*{Specializing from 2-variables to 1-variable}
Let $f$ be a cuspidal $p$-stabilized eigenform obtained by specializing $F$ at a classical height one prime ideal $P_k$ of weight $k \geq 2$.  Let $\rho_f : G_\Sigma \rightarrow \Gl_2(O_f)$ denote the Galois representation associated to $f$ (constructed by Deligne). Here, $O_f$ denotes the integral closure of $\frac{\R_F}{P_k}$. We also obtain a natural map $\pi_f : \R_F \rightarrow O_f$. We let $L_f$ denote the free $O_f$-module of rank $2$ on which $G_\Sigma$ acts to let us obtain the Galois representation $\rho_f$. Let $R_{f,G}$ be the completed tensor product $O_f \hotimes \R_G$. We have natural maps $j_F: R_F \rightarrow O_f \hookrightarrow R_{f,G}$ and $j_G:\R_G \hookrightarrow R_{f,G}$. Consider the $4$-dimensional Galois representation $\rho_1:G_\Sigma\rightarrow\Gl_4(R_{f,G})$ given by the action of $G_\Sigma$ on the rank $4$ free $R_{f,G}$-module: \begin{align*}
L_1 := \Hom_{R_{f,G}} \left(L_F \otimes_{R_F} R_{f,G},\ L_G \otimes_{R_G} \R_{f,G} \right).
\end{align*}

Hida has also constructed an element $\theta_1^{\Sigma_0}$, in the fraction field of $R_{f,G}$, that generates a non-primitive 1-variable Rankin-Selberg $p$-adic $L$-function associated to $\rho_1$. See section 7.4 in \cite{hida1993elementary} for its construction and the precise interpolation properties that $\theta_1^{\Sigma_0}$ satisfies. For the $1$-variable $p$-adic $L$-function, one can vary the weight of $G$. We have a natural map $\pi_{2,1} : \R_{F,G}\rightarrow R_{f,G}$ obtained by the maps $j_F:\R_{F} \rightarrow O_f \hookrightarrow R_{f,G}$ and $j_G:\R_G \hookrightarrow R_{f,G}$.  The map $\pi_{2,1} : \R_{F,G}\rightarrow R_{f,G}$ lets us obtain the following isomorphism of Galois representations:
\begin{align} \pi_{2,1} \circ \rho_2 \cong \rho_1.\end{align}
Hida has proved following theorem relating $\theta_2^{\Sigma_0}$ and $\theta_1^{\Sigma_0}$ (see the comment just before Theorem 2 in section 7.4 of \cite{hida1993elementary}):

\begin{Oldtheorem}[Hida]\label{analysis-2-1} $\pi_{2,1} \left(\theta_2^{\Sigma_0}\right) =  \theta_1^{\Sigma_0}.$
\end{Oldtheorem}
Associated to $\rho_1$, we can construct the Selmer group $\Sel_{\rho_1}(\Q)$ . We will assume that $\Sel_{\rho_1}(\Q)^\vee$ is a torsion $R_{f,G}$-module. We prove the following control theorem relating $\Sel_{\rho_1}(\Q)$ and $\Sel_{\rho_2}(\Q)$:

\begin{Oldtheorem} \label{algebra-2-1}
We have the following equality in the divisor group of $R_{f,G}$~:
\begin{align*}
 \Div \left( {\Sel_{\rho_2}(\Q)}^\vee \otimes_{R_{F,G}} R_{f,G} \right) &+  \Div\bigg(\Tor_1^{R_{F,G}}\left(R_{f,G}, \ H^0(G_\Sigma,D_{\rho_2})^\vee \right)\bigg)   \\ &=  \Div \bigg( \Sel_{\rho_1}(\Q)^\vee \bigg).
 \end{align*}
\end{Oldtheorem}
The main conjecture for $\rho_1$ predicts the following equality of divisors in $R_{f,G}$:
\begin{align} \label{MC1} \tag{MC-$\rho_1$}
 \Div \left(\theta_1^{\Sigma_0} \right)   \stackrel{?}{=}\Div \left(\Sel_{\rho_1}(\Q)^\vee\right) - \Div \left( H^0(G_\Sigma, D_{\rho_1})^\vee \right).
\end{align}
We prove the following theorem to show that Theorem \ref{analysis-2-1} and Theorem \ref{algebra-2-1} are completely consistent with the main conjectures \ref{MC2} and \ref{MC1}:

\begin{Oldtheorem} \label{specialization-2-1}
If the main conjecture \ref{MC2} holds, then the main conjecture \ref{MC1} also holds.
\end{Oldtheorem}

\begin{Remark}
One can combine Proposition \ref{p-regular-prop} and Lemma \ref{p-invertible} to obtain the following equality of divisors in $R_{f,G}$:
\begin{align} \label{tor-h0-iso-2-1}
 &\Div\bigg(\Tor_1^{R_{F,G}}\left(R_{f,G}, \ H^0(G_\Sigma,D_{\rho_2})^\vee \right)\bigg) \\ \notag &=\Div \bigg(\left( H^0(G_\Sigma,D_{\rho_2})^\vee[\ker(\pi_{2,1})]\right) \otimes_{\frac{R_{F,G}} {\ker(\pi_{2,1})}} R_{f,G}\bigg),
\end{align}
which is similar to the one obtained in equation (\ref{tor-h0-iso}). As mentioned in Remark \ref{tor-preference}, we prefer stating Theorem \ref{algebra-2-1} using the term on the left hand side of equation (\ref{tor-h0-iso-2-1})  involving the Tor functor (instead of the term appearing on the right hand side).
\end{Remark}

\begin{Remark}
The main conjectures, formulated in \cite{greenberg1994iwasawa}, relate the primitive Selmer groups to the primitive $p$-adic $L$-functions. It is possible to relate the primitive Selmer groups to the non-primitive Selmer groups; it is also possible to relate the primitive $p$-adic $L$-functions to the non-primitive $p$-adic $L$-functions. As a result, it will be possible to formulate a conjecture, that is equivalent to the one formulated in \cite{greenberg1994iwasawa}, relating the non-primitive Selmer groups to the non-primitive $p$-adic $L$-functions. It is the latter formulation that we adopt in this paper. Since we are motivated by Hida's results involving the non-primitive $p$-adic $L$-functions, we will only work with the non-primitive Selmer groups; and since Hida's results simply serve as heuristics to motivate our theorems, we will not concern ourselves with proving that Greenberg's formulation of the main conjectures for $\rho_3$, $\rho_2$ and $\rho_1$ involving primitive Selmer groups and primitive $p$-adic $L$-functions in \cite{greenberg1994iwasawa} are equivalent to \ref{MC3}, \ref{MC2}, \ref{MC1} respectively.
\end{Remark}

\begin{Remark}
We would like make the reader aware of a point that was conveyed to us by Haruzo Hida. The $p$-adic $L$-functions $\theta^{\Sigma_0}_1$, $\theta^{\Sigma_0}_2$ and $\theta^{\Sigma_0}_3$ are ``genuine'' in the sense of \cite{hida1996search} and do not exactly equal the $p$-adic $L$-functions constructed in \cite{hida1993elementary}. Recall that we have inclusions $\R_G \hookrightarrow R_{f,G}$ and $\R_G \hookrightarrow \R_{F,G} \hookrightarrow \R_{F,G}[[\Gamma]]$. There is an element $\Theta \in \R_G$ that corresponds  to the divisor of a one-variable $p$-adic $L$-function associated to the 3-dimensional adjoint representation $\Ad^0(\rho_G)$ (see Conjecture 1.0.1 in \cite{hida1996search}).
The $p$-adic $L$-functions, that are constructed in \cite{hida1993elementary} as elements of the fraction fields of $R_{f,G}$, $R_{F,G}$ and $R_{F,G}[[\Gamma]]$ respectively, are in fact equal to $\frac{\theta^{\Sigma_0}_1}{\Theta}$, $\frac{\theta^{\Sigma_0}_2}{\Theta}$ and $\frac{\theta^{\Sigma_0}_3}{\Theta}$ respectively. A discussion surrounding the need to introduce this modification, which is related to the choice of a certain period (Neron period versus a period involving the Peterson inner product), is carefully explained in \cite{hida1996search} (see the introduction and Section 6 there). This point will not matter to us since the maps $\pi_{3,2}$ and $\pi_{2,1}$ are $R_G$-linear.
\end{Remark}

\begin{Remark} \label{obstacle}
The obstacle, presented by the fact the local domains appearing in Hida theory need not even be UFDs, manifests itself in the following way. Although one is interested in results, such as the one in Theorem \ref{algebra-3-2} (and Theorem \ref{algebra-2-1} respectively), in the divisor group of $R_{F,G}$ (and $R_{f,G}$ respectively), one is forced to consider the localizations of $\Sel_{\rho_3}(\Q)^\vee$ (and $\Sel_{\rho_2}(\Q)^\vee$ respectively) at certain height two prime ideals of $R_{F,G}[[\Gamma]]$ (and $R_{F,G}$ respectively).
\end{Remark}

\begin{Remark}
We always work with the hypothesis (labeled \ref{tor} in Section \ref{selmer}) that the $R_{F,G}[[\Gamma]]$-module $\Sel_{\rho_3}(\Q)^\vee$, the $R_{F,G}$-module $\Sel_{\rho_2}(\Q)^\vee$ and the $R_{f,G}$-module $\Sel_{\rho_1}(\Q)^\vee$ are torsion. To establish the hypothesis \ref{tor},  one can utilize the techniques involving Euler systems developed by Lei-Loeffler-Zerbes \cite{lei2012euler} and Kings-Loeffler-Zerbes \cite{kings2015rankin}. We refer the reader to \cite{lei2012euler} and  \cite{kings2015rankin} for the precise conditions when the Euler system machinery is available. Under these conditions, to establish \ref{tor}, one has to show that Hida's Rankin-Selberg $p$-adic $L$-functions $\theta^{\Sigma_0}_{3}$, $\theta^{\Sigma_0}_{2}$ and $\theta^{\Sigma_0}_{1}$ are non-zero. The non-vanishing of these $p$-adic $L$-functions is a recent work in progress of Jeanine Van Order. See \cite{2012arXiv1207.1672V} and  \cite{2012arXiv1207.1673V}. \\

It will be possible to obtain results, consistent with the various main conjectures, without assuming \ref{tor}. We briefly indicate how to obtain such results, though we leave the precise details to the interested reader. Suppose the $R_{F,G}$-module $\Sel_{\rho_2}(\Q)^\vee$ is not torsion. The Iwasawa main conjecture predicts that $\theta^{\Sigma_0}_{2}$ equals zero.  On the analytic side, Hida's theorem (Theorem \ref{analysis-3-2}) predicts that $\pi_{3,2}(\theta^{\Sigma_0}_{3})$ equals zero. On the algebraic side, one can establish a control theorem to show that $\ker(\pi_{3,2})$ belongs to the support of the $R_{F,G}[[\Gamma]]$-module $\Sel_{\rho_3}(\Q)^\vee$.

Suppose the $R_{f,G}$-module $\Sel_{\rho_1}(\Q)^\vee$ is not torsion. The Iwasawa main conjecture predicts that $\theta^{\Sigma_0}_{1}$ equals zero.  On the analytic side, Hida's theorem (Theorem \ref{analysis-2-1}) predicts that $\pi_{2,1}(\theta^{\Sigma_0}_{2})$ equals zero. On the algebraic side, one can similarly establish a control theorem to show that $\ker(\pi_{2,1})$ belongs to the support of the $R_{F,G}$-module $\Sel_{\rho_2}(\Q)^\vee$.
\end{Remark}

\section{Selmer groups} \label{selmer}

The hypothesis \ref{p-distinguished-inertia} lets us obtain following $\Gal{\overline{\Q}_p}{\Q_p}$-equivariant short exact sequences of free $R_F$ and $R_G$ modules respectively:
\begin{align} \label{seq-dim-2}
0 \rightarrow \underbrace{\Fil^+L_F}_{\text{Rank}=1} \rightarrow \underbrace{L_F}_{\text{Rank}=2} \rightarrow \underbrace{\frac{L_F}{\Fil^+L_F}}_{\text{Rank}=1} \rightarrow 0, \qquad 0 \rightarrow \underbrace{\Fil^+L_G}_{\text{Rank}=1} \rightarrow \underbrace{L_G}_{\text{Rank}=2} \rightarrow \underbrace{\frac{L_G}{\Fil^+L_G}}_{\text{Rank}=1} \rightarrow 0.
\end{align}
Here, the action of $\Gal{\overline{\Q}_p}{\Q_p}$ on the rank $1$ modules $\Fil^+L_F$ and $\Fil^+L_G$ is given by ramified characters $\delta_F : \Gal{\overline{\Q}_p}{\Q_p} \rightarrow R_F^\times$ and $\delta_G : \Gal{\overline{\Q}_p}{\Q_p} \rightarrow R_G^\times$ respectively. The action of $\Gal{\overline{\Q}_p}{\Q_p}$ on the rank $1$ modules $\frac{L_F}{\Fil^+L_F}$ and $\frac{L_G}{\Fil^+L_G}$ is given by unramified characters $\epsilon_F : \Gal{\overline{\Q}_p}{\Q_p} \rightarrow R_F^\times$ and $\epsilon_G : \Gal{\overline{\Q}_p}{\Q_p} \rightarrow R_G^\times$ respectively. Note that $\epsilon_F(\Frob_p)=a_p$ and $\epsilon_G(\Frob_p)=b_p$. Furthermore, using the residual characters $\overline{\delta}_F : \Gal{\overline{\Q}_p}{\Q_p} \rightarrow \overline{\mathbb{F}}_p^\times$ and $\overline{\delta}_G : \Gal{\overline{\Q}_p}{\Q_p} \rightarrow \overline{\mathbb{F}}_p^\times$ associated to $\delta_F$ and $\delta_G$ respectively, the hypothesis \ref{p-distinguished-inertia} can be restated in the following way:
\begin{align*}
\tag{$p$-DIS-IN} \overline{\delta}_F \mid_{I_p} \neq \mathbf{1}, \qquad \overline{\delta}_G \mid_{I_p} \neq \mathbf{1}.
\end{align*}
Here, $I_p$ denotes the inertia subgroup inside the decomposition group $\Gal{\overline{\Q}_p}{\Q_p}$. Analogous to the short exact sequences in (\ref{seq-dim-2}), we can now form the following $\Gal{\overline{\Q}_p}{\Q_p}$-equivariant short exact sequence of free $R_{F,G}[[\Gamma]]$-modules:
\begin{align*}
0 \rightarrow \underbrace{\Fil^+L_3}_{\text{Rank}=2} \rightarrow \underbrace{L_3}_{\text{Rank}=4} \rightarrow \underbrace{\frac{L_3}{\Fil^+L_3}}_{\text{Rank}=2} \rightarrow 0,
\end{align*}
where
\begin{align*}
&  \qquad  \Fil^+ L_3 : = \Hom_{R_{F,G}} \left(L_F \otimes_{R_F} R_{F,G} , \ \Fil^+ L_G \otimes_{R_G} R_{F,G} \right) \otimes_{R_{F,G}} R_{F,G}[[\Gamma]](\kappa^{-1}).
\end{align*}
We would like to emphasize the following isomorphisms of $R_{F,G}$ and $R_{f,G}$-modules respectively.
\begin{align}
L_2 \cong  L_3 \otimes_{R_{F,G}[[\Gamma]]} R_{F,G}, \qquad  \qquad L_1 \cong  L_2 \otimes_{R_{F,G}} R_{f,G}.
\end{align}
Once again, analogous to the short exact sequences in (\ref{seq-dim-2}), we can now form the following $\Gal{\overline{\Q}_p}{\Q_p}$-equivariant short exact sequence of free $R_{F,G}$ and $R_{f,G}$ modules respectively:
\begin{align*}
0 \rightarrow \underbrace{\Fil^+L_2}_{\text{Rank}=2} \rightarrow \underbrace{L_2}_{\text{Rank}=4} \rightarrow \underbrace{\frac{L_2}{\Fil^+L_2}}_{\text{Rank}=2} \rightarrow 0, \quad 0 \rightarrow \underbrace{\Fil^+L_1}_{\text{Rank}=2} \rightarrow \underbrace{L_1}_{\text{Rank}=4} \rightarrow \underbrace{\frac{L_1}{\Fil^+L_1}}_{\text{Rank}=2} \rightarrow 0,
\end{align*}
where
\begin{align*}
&  \qquad && \Fil^+ L_2  := \Fil^+L_3 \otimes_{R_{F,G}[[\Gamma]]} R_{F,G}, \qquad  \Fil^+ L_1 := \Fil^+ L_2 \otimes_{R_{F,G}} R_{f,G}.
\end{align*}
To the Galois representations $\rho_3$, $\rho_2$ and $\rho_1$, we can now associate the following discrete modules:
\begin{align*}
&D_{\rho_3} = L_3 \otimes_{R_{F,G}[[\Gamma]]} \hat{R_{F,G}[[\Gamma]]}, \quad && \Fil^+D_{\rho_3} = \Fil^+L_3 \otimes_{R_{F,G}[[\Gamma]]} \hat{R_{F,G}[[\Gamma]]} \\ & D_{\rho_2} = L_2 \otimes_{R_{F,G}} \hat{R_{F,G}}, \quad  &&\Fil^+D_{\rho_2} = \Fil^+L_2 \otimes_{R_{F,G}} \hat{R_{F,G}} \\ &D_{\rho_1} = L_1 \otimes_{R_{f,G}} \hat{R_{f,G}}, \quad &&\Fil^+D_{\rho_1} = \Fil^+L_1 \otimes_{R_{f,G}} \hat{R_{f,G}}.
\end{align*}

We now define the (discrete) non-primitive Selmer groups associated to the Galois representations $\rho_3$, $\rho_2$ and $\rho_1$ respectively. Since we do not discuss primitive Selmer groups in this paper, we will refer to the non-primitive Selmer groups simply as Selmer groups (without the prefix ``non-primitive''). Let $i \in \{1,2,3\}$.
\begin{align*}
\Sel_{\rho_i}(\Q) &:= \ker \left( H^1(G_\Sigma, D_{\rho_i}) \xrightarrow {\phi_{\rho_i}^{\Sigma_0}} H^1\left(I_p,\frac{D_{\rho_i}}{\Fil^+D_{\rho_i}} \right)^{\Gamma_p} \right).
\end{align*}
Here, $\Gamma_p = \frac{\Gal{\overline{\Q}_p}{\Q_p}}{I_p}$. Throughout this paper, we will suppose that the following hypothesis holds.
\begin{align} \label{tor}
\tag{$\mathrm{TOR}$} \text{Rank}_{R_{F,G}[[\Gamma]]} \left( \Sel_{\rho_3}(\Q)^\vee \right) &=0, \\ \notag \text{Rank}_{R_{F,G}} \left( \Sel_{\rho_2}(\Q)^\vee \right) &=0, \\ \notag \text{Rank}_{R_{f,G}} \left( \Sel_{\rho_1}(\Q)^\vee \right) &=0.
\end{align}

Let us now discuss the proofs of Theorems \ref{algebra-3-2} and \ref{specialization-3-2} along with the proofs of  Theorems \ref{algebra-2-1} and \ref{specialization-2-1}. These proofs will utilize the control theorems and specialization results developed in our earlier work \cite{bharathwaj2016algebraic}.  The control theorems are developed in Section 6 of \cite{bharathwaj2016algebraic} (see Proposition 6.2 in \cite{bharathwaj2016algebraic}). The specialization results are developed Section 5 of \cite{bharathwaj2016algebraic} (see Proposition 5.2 in \cite{bharathwaj2016algebraic}). We will treat those results as black boxes. We summarize the results of \cite{bharathwaj2016algebraic} in section \ref{towards-the-proof}.

\section{Proofs of Theorem \ref{algebra-3-2} and \ref{specialization-3-2}} \label{proof-3-2}

One can use Proposition 6.2 in Section 6 of \cite{bharathwaj2016algebraic} to establish a control theorem relating $\Sel_{\rho_3}(\Q)$ to $\Sel_{\rho_2}(\Q)$ (see Section \ref{towards-the-proof}).  We have the following equality of divisors in $R_{F,G}$:
 \begin{align} \label{control-eqn-previous-3-2}
&\Div \left( \Sel_{\rho_3}(\Q)^\vee \otimes_{R_{F,G}[[\Gamma]]} \R_{F,G} \right) + \Div \bigg( \Tor_{1}^{R_{F,G}[[\Gamma]]} \left( R_{F,G}, H^0(G_{\Sigma,D_{\rho_3}})^\vee \right)\bigg)  \\ = \notag \  & \Div \left( \Sel_{\rho_2}(\Q)^\vee \right) + \Div \left( \Tor_1^{R_{F,G}[[\Gamma]]} \left( R_{F,G}, H^0\left(I_p,\frac{D_{\rho_3}}{\Fil^+D_{\rho_3}}\right)^\vee \right)_{\Gamma_p} \right) .
\end{align}

We will view $R_{F,G}$ as an $R_{F,G}[[\Gamma]]$-module via the map $\pi_{3,2}:R_{F,G}[[\Gamma]] \rightarrow R_{F,G}$. As a topological group, $\Gamma \cong \Z_p$. Let's choose a topological generator $\gamma_0$ of $\Gamma$. The prime ideal $\ker(\pi_{3,2})$ is generated by $\gamma_0-1$ as an $R_{F,G}[[\Gamma]]$-module. Consider the short exact sequence $0 \rightarrow R_{F,G}[[\Gamma]] \xrightarrow {\gamma_0-1}  R_{F,G}[[\Gamma]] \rightarrow R_{F,G} \rightarrow 0$ of $R_{F,G}[[\Gamma]]$-modules. Taking the tensor product with the $R_{F,G}[[\Gamma]]$-module $H^0(G_{\Sigma},D_{\rho_3})^\vee$ over the ring $R_{F,G}[[\Gamma]]$ gives us following isomorphism:
\begin{align} \label{tor-iso-3-2}
\Tor_{1}^{R_{F,G}[[\Gamma]]} \left( R_{F,G}, \ H^0(G_{\Sigma},D_{\rho_3})^\vee \right) \cong H^0(G_\Sigma, D_{\rho_3})^\vee[\ker(\pi_{3,2})].
\end{align}
Lemma \ref{loc-tor-3-2} establishes the following isomorphism of $R_{F,G}$-modules:
\begin{align} \label{tor-iso-local-3-2}
\Tor_1^{R_{F,G}[[\Gamma]]} \left( R_{F,G}, \ H^0\left(I_p,\frac{D_{\rho_3}}{\Fil^+D_{\rho_3}}\right)^\vee \right)_{\Gamma_p} \cong \frac{R_{F,G}}{\left(1-\frac{i_F(a_p)}{i_G(b_p)}\right)}.
\end{align}
Combining equations (\ref{control-eqn-previous-3-2}), (\ref{tor-iso-3-2}), and (\ref{tor-iso-local-3-2}) completes the proof of Theorem \ref{algebra-3-2}. \\

Now suppose that the main conjecture \ref{MC3} holds. That is, we have the following equality in the divisor group of $R_{F,G}[[\Gamma]]$:
\begin{align} \tag{MC-$\rho_3$}  \Div \left(\theta_3^{\Sigma_0} \right)   =\Div \left(\Sel_{\rho_3}(\Q)^\vee\right) - \Div \left( H^0(G_\Sigma, D_{\rho_3})^\vee \right).
\end{align}
We will later outline in Section \ref{towards-the-proof} how the specialization result (Proposition 5.2 in Section 5 of \cite{bharathwaj2016algebraic}) lets us obtain the following equality of divisors in $R_{F,G}$:
\begin{align} \label{specialization-eqn-3-2}  \Div \left(\pi_{3,2}\left(\theta_3^{\Sigma_0} \right)\right) & =\Div \left(\Sel_{\rho_3}(\Q)^\vee \otimes_{R_{F,G}[[\Gamma]]} R_{F,G}\right) + \\ \notag & +  \Div \bigg( \Tor_{1}^{R_{F,G}[[\Gamma]]} \left( R_{F,G}, H^0(G_{\Sigma},D_{\rho_3})^\vee \right)\bigg)  - \Div \left( H^0(G_\Sigma, D_{\rho_2})^\vee \right).
\end{align}
Theorem \ref{analysis-3-2} establishes the following equality in the divisor group of $R_{F,G}$:
\begin{align} \label{analysis-eqn-3-2}
\Div \left(\pi_{3,2}\left(\theta_3^{\Sigma_0} \right)\right)   = \Div\left(\theta^{\Sigma_0}_{\rho_2}\right) + \Div \left( 1-\frac{i_F(a_p)}{i_G(b_p)}\right).
\end{align}
Theorem \ref{algebra-3-2} along with equation (\ref{tor-iso-3-2}) establishes the following equality in the divisor group of $R_{F,G}$:
\begin{align} \label{algebra-eqn-3-2}
 \Div \left( \Sel_{\rho_3}(\Q)^\vee \otimes_{R_{F,G}[[\Gamma]]} \R_{F,G} \right) &+ \Div \bigg(\Tor_{1}^{R_{F,G}[[\Gamma]]} \left( R_{F,G}, \ H^0(G_{\Sigma},D_{\rho_3})^\vee \right)\bigg)   \\ \notag &=  \Div \left( \Sel_{\rho_2}(\Q)^\vee \right) + \Div \left( 1-\frac{i_F(a_p)}{i_G(b_p)} \right).
\end{align}

Now combining equations (\ref{specialization-eqn-3-2}), (\ref{analysis-eqn-3-2}) and (\ref{algebra-eqn-3-2}), we have the following equality of divisors in $R_{F,G}$:
\begin{align}  \tag{MC-$\rho_2$}
\Div \left(\theta_2^{\Sigma_0} \right)   =\Div \left(\Sel_{\rho_2}(\Q)^\vee\right) - \Div \left( H^0(G_\Sigma, D_{\rho_2})^\vee \right).
\end{align}
This completes the proof of Theorem \ref{specialization-3-2}.

\section{Proofs of Theorem \ref{algebra-2-1} and \ref{specialization-2-1}} \label{proof-2-1}

The proofs of  Theorem \ref{algebra-2-1} and \ref{specialization-2-1} will be similar to the proofs of  Theorem \ref{algebra-3-2} and \ref{specialization-3-2}. Proposition 6.2 in Section 6 of \cite{bharathwaj2016algebraic} will allow us to establish a control theorem relating $\Sel_{\rho_2}(\Q)^\vee \otimes_{R_{F,G}}R_{f,G}$ to $\Sel_{\rho_1}(\Q)^\vee$.  We have the following equality of divisors in $R_{f,G}$:
\begin{align} \label{control-eqn-previous-2-1}
&\Div \left( \Sel_{\rho_2}(\Q)^\vee \otimes_{R_{F,G}} \R_{f,G} \right) + \Div \bigg( \Tor_{1}^{R_{F,G}} \left( R_{f,G}, H^0(G_{\Sigma,D_{\rho_2}})^\vee \right)\bigg)  \\ = \notag \  & \Div \left( \Sel_{\rho_1}(\Q)^\vee \right) + \Div \left( \Tor_1^{R_{F,G}} \left( R_{f,G}, H^0\left(I_p,\frac{D_{\rho_2}}{\Fil^+D_{\rho_2}}\right)^\vee \right)_{\Gamma_p} \right) .
\end{align}

Lemma \ref{local-tor-torsion} establishes the following equality:
\begin{align} \label{tor-iso-local-2-1}
\Tor_1^{R_{F,G}} \left( R_{f,G}, \ H^0\left(I_p,\frac{D_{\rho_2}}{\Fil^+D_{\rho_2}}\right)^\vee \right) =0.
\end{align}
Combining equations (\ref{control-eqn-previous-2-1}) and (\ref{tor-iso-local-2-1}) completes the proof of Theorem \ref{algebra-2-1}. \\

Now suppose that the main conjecture \ref{MC2} holds. That is, we have the following equality in the divisor group of $R_{F,G}$:
\begin{align} \tag{MC-$\rho_2$}  \Div \left(\theta_2^{\Sigma_0} \right)   =\Div \left(\Sel_{\rho_2}(\Q)^\vee\right) - \Div \left( H^0(G_\Sigma, D_{\rho_2})^\vee \right).
\end{align}
As before, we will later outline in Section \ref{towards-the-proof} how the specialization result (Proposition 5.2 in Section 5 of \cite{bharathwaj2016algebraic}) lets us obtain the following equality of divisors in $R_{f,G}$:
\begin{align} \label{specialization-eqn-2-1} \Div \left(\pi_{2,1}\left(\theta_2^{\Sigma_0} \right)\right)  &= \Div \left(\Sel_{\rho_2}(\Q)^\vee \otimes_{R_{F,G}} R_{f,G}\right)  \\ \notag & +  \Div \bigg( \Tor_{1}^{R_{F,G}} \left( R_{f,G}, H^0(G_{\Sigma},D_{\rho_2})^\vee \right)\bigg)   - \Div \left( H^0(G_\Sigma, D_{\rho_1})^\vee \right).
\end{align}
Theorem \ref{analysis-2-1} establishes the following equality in the divisor group of $R_{f,G}$:
\begin{align} \label{analysis-eqn-2-1}
\Div \left(\pi_{2,1}\left(\theta_2^{\Sigma_0} \right)\right)   = \Div\left(\theta^{\Sigma_0}_{\rho_1}\right).
\end{align}
Theorem \ref{algebra-2-1} establishes the following equality in the divisor group of $R_{f,G}$:
\begin{align} \label{algebra-eqn-2-1}
 \Div \left( \Sel_{\rho_2}(\Q)^\vee \otimes_{R_{F,G}} \R_{f,G} \right) &+ \Div \bigg(\Tor_{1}^{R_{F,G}} \left( R_{F,G}, \ H^0(G_{\Sigma},D_{\rho_2})^\vee \right)\bigg)  \\ & \notag =  \Div \left( \Sel_{\rho_1}(\Q)^\vee \right).
\end{align}

Now combining equations (\ref{specialization-eqn-2-1}), (\ref{analysis-eqn-2-1}) and (\ref{algebra-eqn-2-1}), we have the following equality of divisors in $R_{f,G}$:
\begin{align}  \tag{MC-$\rho_1$}
\Div \left(\theta_1^{\Sigma_0} \right)   =\Div \left(\Sel_{\rho_1}(\Q)^\vee\right) - \Div \left( H^0(G_\Sigma, D_{\rho_1})^\vee \right).
\end{align}
This completes the proof of Theorem \ref{specialization-2-1}.

\section{Towards the control theorems and specialization results} \label{towards-the-proof}

The hypothesis \ref{tor} plays a crucial role in establishing many structural properties of the Galois cohomology groups and Selmer groups. The control theorems and specialization results rely heavily on Greenberg's foundational works (\cite{MR2290593}, \cite{greenberg2010surjectivity} and \cite{greenberg2014pseudonull}). In these works, the Weak Leopoldt conjecture often comes into play. Let $i \in \{1,2,3\}$. We consider the following group:
\begin{align*}
\Sha^2\left(D_{\rho_i}\right) := \ker\left(H^2(G_\Sigma,D_{\rho_i}) \rightarrow \prod_{\nu \in \Sigma} H^2(\Gal{\overline{\Q}_\nu}{\Q_\nu},D_{\rho_i}) \right).
\end{align*}
The validity of \ref{tor} ensures us that the following statement  \ref{weak-leo} holds:
\begin{enumerate}[style=sameline, leftmargin=2.5cm, style=sameline, align=left,label=\textsc{Weak Leo}, ref=\textsc{Weak Leo}]
\item \label{weak-leo} $\Sha^2(D_{\rho_3})^\vee$, $\Sha^2(D_{\rho_2})^\vee$ and $\Sha^2(D_{\rho_1})^\vee$  are torsion $R_{F,G}[[\Gamma]]$, $R_{F,G}$ and $R_{f,G}$ modules respectively.
\end{enumerate}
 For each $i \in \{1,2,3\}$ and each $\nu \in \Sigma_0$, the validity of \ref{tor} also lets us deduce the following equalities:
\begin{align*}
H^2(\Gal{\overline{\Q}_\nu}{\Q_\nu},D_{\rho_i})=0, \ \ \forall \nu \in \Sigma_0; \qquad H^2\left(\Gal{\overline{\Q}_p}{\Q_p},\frac{D_{\rho_i}}{\Fil^+D_{\rho_i}}\right)=0.\end{align*}
The Galois representation $\rho_3$ is related to the
cyclotomic deformation of $\rho_2$ (see Section 3 of \cite{greenberg1994iwasawa} for a description of the cyclotomic deformation). Equation (\ref{more-h2-vanishing-cyc}) can be deduced from the arguments given in Section 5 of \cite{greenberg2010surjectivity}. The arguments involve combining local duality, Proposition 3.10 in \cite{MR2290593} along with results from Section 3 in \cite{greenberg1994iwasawa}. These arguments are also described in Section 4 of \cite{bharathwaj2016algebraic}.
\begin{align} \label{more-h2-vanishing-cyc}
H^2\left(\Gal{\overline{\Q}_p}{\Q_p},D_{\rho_3}\right) = H^2\left(\Gal{\overline{\Q}_p}{\Q_p},\Fil^+D_{\rho_3}\right)=0.
\end{align}
We  will use local duality to establish a similar result for $\rho_2$ and $\rho_1$. See Lemma \ref{van-h2-rho-2-1}.
 \begin{align*}
H^2\left(\Gal{\overline{\Q}_p}{\Q_p},D_{\rho_2}\right) = H^2\left(\Gal{\overline{\Q}_p}{\Q_p},\Fil^+D_{\rho_2}\right)=0, \\
H^2\left(\Gal{\overline{\Q}_p}{\Q_p},D_{\rho_1}\right) = H^2\left(\Gal{\overline{\Q}_p}{\Q_p},\Fil^+D_{\rho_1}\right)=0.
\end{align*}
These observations, along with Corollary \ref{global-h2-vanish}, establish the following statement:
\begin{enumerate}[style=sameline, leftmargin=2cm, style=sameline, align=left,label=\textsc{Van-H2}, ref=\textsc{Van-H2}]
\item \label{van-h2} Let $i \in \{1,2,3\}$. Let $\nu \in \Sigma_0$. We have
\begin{align*}
H^2(G_\Sigma,D_{\rho_i}) & =H^2(\Gal{\overline{\Q}_\nu}{\Q_\nu},D_{\rho_i}) = H^2\left(\Gal{\overline{\Q}_p}{\Q_p},D_{\rho_i}\right)= \\ &= H^2\left(\Gal{\overline{\Q}_p}{\Q_p},\Fil^+D_{\rho_i}\right) = H^2\left(\Gal{\overline{\Q}_p}{\Q_p},\frac{D_{\rho_i}}{\Fil^+D_{\rho_i}}\right)=0.
\end{align*}
\end{enumerate}

We will now summarize the arguments involved in deducing the specialization results given in equations (\ref{specialization-eqn-3-2}) and (\ref{specialization-eqn-2-1}). Let $X_{3}=\Sel_{\rho_3}(\Q)^\vee$ and $X_2=\Sel_{\rho_2}(\Q)^\vee$. Given the divisor $\Div(X_3)$ in $R_{F,G}[[\Gamma]]$  (and $\Div(X_2)$ in $R_{F,G}$ respectively), we would like to find the divisor $\Div\left(X_3 \otimes_{R_{F,G}[[\Gamma]]} R_{F,G}\right)$ in $R_{F,G}$ (and $\Div\left(X_2 \otimes_{R_{F,G}}R_{f,G}\right)$ in  $R_{f,G}$ respectively). For this purpose, we will use Proposition 5.2 in Section 5 of \cite{bharathwaj2016algebraic}. One of the hypotheses involved in the Proposition there concerns the following statement:

\begin{enumerate}[style=sameline, leftmargin=2cm, style=sameline, align=left,label=\textsc{Fin-Proj}, ref=\textsc{Fin-Proj}]
\item \label{fin-proj} For every height two  prime ideal $Q_{3,2}$ in $R_{F,G}[[\Gamma]]$ containing $\ker(\pi_{3,2})$, the projective dimension of $\left(\Sel_{\rho_3}(\Q)^\vee\right)_{Q_{3,2}}$ is finite. \\
For every height two  prime ideal $Q_{2,1}$ in $R_{F,G}$ containing $\ker(\pi_{2,1})$, the projective dimension of $\left(\Sel_{\rho_2}(\Q)^\vee\right)_{Q_{2,1}}$ is finite.
\end{enumerate}

Let us suppose that \ref{fin-proj} holds (we will briefly discuss later why \ref{fin-proj} holds). Another hypothesis involved in Proposition 5.2 in section 5 of \cite{bharathwaj2016algebraic} requires us to show that the maximal $R_{F,G}[[\Gamma]]$ (and $R_{F,G}$ respectively) pseudo-null submodule of $X_3$ (and $X_2$ respectively) is trivial. For this purpose, we will first define the strict (non-primitive) Selmer group, for each $i \in \{1,2,3\}$, as follows:
\begin{align*}
\text{Sel}^{\Sigma_0,str}_{\rho_i}(\Q) &:= \ker \left( H^1(G_\Sigma, D_{\rho_i}) \xrightarrow {\phi_{\rho_i}^{\Sigma_0,str}} H^1\left(\Gal{\overline{\Q}_p}{\Q_p},\frac{D_{\rho_i}}{\Fil^+D_{\rho_i}} \right) \right).
\end{align*}
Combining Proposition 4.2.1 and Proposition 4.3.2 from \cite{greenberg2014pseudonull} with the statements \ref{tor} and \ref{van-h2}, we obtain that the $R_{F,G}[[\Gamma]]$-module $\text{Sel}^{\Sigma_0,str}_{\rho_3}(\Q)^\vee$ and the $R_{F,G}$-module $\text{Sel}^{\Sigma_0,str}_{\rho_2}(\Q)^\vee$ have no non-trivial pseudo-null submodules. Also for each $i \in \{1,2,3\}$, one has the following short exact sequence relating $\text{Sel}^{\Sigma_0,str}_{\rho_i}(\Q)$ to $\Sel_{\rho_i}(\Q)$:
\begin{align*}
0 \rightarrow  \text{Sel}^{\Sigma_0,str}_{\rho_i}(\Q) \xrightarrow {\alpha_i} \Sel_{\rho_i}(\Q) \rightarrow H^1\left(\Gamma_p,H^0\left(I_p,\frac{D_{\rho_i}}{\Fil^+D_{\rho_i}}\right)\right)\rightarrow  0.
\end{align*}
In the cases we are interested in, Lemma \ref{strict-selmer-difference} tells us that $H^1\left(\Gamma_p,H^0\left(I_p,\frac{D_{\rho_i}}{\Fil^+D_{\rho_i}}\right)\right) =0$, for each $i \in \{1,2,3\}$. This lets us conclude that the map $\alpha_i$ is an isomorphism for each $i \in \{1,2,3\}$. We have the following statement:

\begin{enumerate}[style=sameline, leftmargin=2cm, style=sameline, align=left,label=\textsc{No-PN}, ref=\textsc{No-PN}]
\item \label{no-pn} The $R_{F,G}[[\Gamma]]$-module $\Sel_{\rho_3}(\Q)^\vee$ has no non-trivial pseudo-null submodules.\\
The $R_{F,G}$-module $\Sel_{\rho_2}(\Q)^\vee$ has no non-trivial pseudo-null submodules.
\end{enumerate}
Let $Q_{3,2}$ (and $Q_{2,1}$ respectively) be a height two prime ideal in $R_{F,G}[[\Gamma]]$ (and $R_{F,G}$ respectively) containing $\ker(\pi_{3,2})$ (and $\ker(\pi_{2,1})$ respectively). \ref{no-pn} allows us to conclude that
\begin{align} \label{depth}
\depth_{\left(R_{F,G}[[\Gamma]]\right)_{Q_{3,2}}} \ (X_3)_{Q_{3,2}} \geq 1, \ \  \depth_{\left(R_{F,G}\right)_{Q_{2,1}}} \ (X_2)_{Q_{2,1}} \geq 1.
\end{align}

The hypotheses \ref{fin-proj} allows us to use the Auslander-Buchsbaum formula. Combining equation (\ref{depth}) along with the Auslander-Buchsbaum formula  over the two-dimensional local rings $\left(R_{F_{G,}}[[\Gamma]]\right)_{Q_{3,2}}$ and $\left(R_{F,G}\right)_{Q_{2,1}}$ allows us to make the following conclusion involving projective~dimensions:
\begin{align} \label{pd-auslaunder}
\pd_{\left(R_{F,G}[[\Gamma]\right)_{Q_{3,2}}} (X_3)_{Q_{3,2}} \leq 1, \qquad \pd_{\left(R_{F,G}\right)_{Q_{2,1}}} (X_2)_{Q_{2,1}} \leq 1.
\end{align}

This allows us to write down projective resolutions for $\left(X_3\right)_{Q_{3,2}}$ (and $\left(X_2\right)_{Q_{2,1}}$ respectively) over $\left(R_{F,G}[[\Gamma]]\right)_{Q_{3,2}}$ (and $\left(R_{F,G}\right)_{Q_{2,1}}$ respectively). Let $A_3 = \left(R_{F,G}[[\Gamma]]\right)_{Q_{3,2}}$ and $A_2=\left(R_{F,G}\right)_{Q_{2,1}}$. Both the rings $A_3$ and $A_2$ have Krull dimension two. We have the following short exact sequences:

\begin{align*}
0 \rightarrow A_3^{n_3} \xrightarrow {\omega_3} A_3^{n_3} \rightarrow \left(X_3\right)_{\q_3} \rightarrow 0, \quad 0 \rightarrow A_2^{n_2} \xrightarrow {\omega_2} A_2^{n_2} \rightarrow \left(X_2\right)_{\q_2} \rightarrow 0.
\end{align*}

Here, we have matrices $\omega_3 \in M_{n_3}(A_3)$ and $\omega_2 \in M_{n_2}(A_2)$. Consider the rings $B_2 := A_3~\otimes_{R_{F,G}[[\Gamma]]}~R_{F,G}$ and $B_1 := A_2 \otimes_{R_{F,G}} R_{f,G}$, both of which have Krull dimension one. Taking the tensor product with the $R_{F,G}[[\Gamma]]$-module $R_{F,G}$ (and $R_{F,G}$-module $R_{f,G}$ respectively), we obtain the following short exact sequences:
\begin{align*}
& 0 \rightarrow B_2^{n_3} \xrightarrow {\pi_{3,2}(\omega_3)} B_2^{n_3} \rightarrow X_3 \otimes_{R_{F,G}[[\Gamma]]} B_2 \rightarrow 0, \\ & 0 \rightarrow B_1^{n_2} \xrightarrow {\pi_{2,1}(\omega_2)} B_1^{n_2} \rightarrow X_2 \otimes_{R_{F,G}}B_1 \rightarrow 0.
\end{align*}

If we let $\len_A(M)$ denote the length of an $A$-module $M$, we have
\begin{align*}
\len_{B_2} \left(\frac{B_2}{(\pi_{3,2}(\det(\omega_3)))}\right)  &= \len_{B_2} \left(X_3 \otimes_{R_{F,G}[[\Gamma]]} B_2\right), \\ \len_{B_1} \left(\frac{B_1}{(\pi_{2,1}(\det(\omega_2)))}\right)  &= \len_{B_1} \left(X_2 \otimes_{R_{F,G}} B_1\right).
\end{align*}

It is this observation that we will maneuver to obtain information about the divisors $\Div\left(X_3 \otimes_{R_{F,G}[[\Gamma]]} R_{F,G}\right)$ and $\Div\left(X_2 \otimes_{R_{F,G}} R_{f,G}\right)$ in $R_{F,G}$ and $R_{f,G}$ respectively. \\

In what follows, we will keep the isomorphism $\Sel_{\rho_i}(\Q_\infty) \cong \text{Sel}^{\Sigma_0,str}_{\rho_i}(\Q_\infty)$ in mind to establish~\ref{fin-proj}. We can combine Corollary 3.2.3 in \cite{greenberg2010surjectivity} with \ref{tor} and \ref{van-h2} to establish the following~statement:
\begin{enumerate}[style=sameline, leftmargin=2cm, style=sameline, align=left,label=\textsc{Surj}, ref=\textsc{Surj}]
\item \label{surj} The global-to-local maps $\phi^{\Sigma_0}_{\rho_3}$,  $\phi^{\Sigma_0}_{\rho_2}$ and $\phi^{\Sigma_0}_{\rho_1}$ defining the Selmer groups $\Sel_{\rho_3}(\Q)$, $\Sel_{\rho_2}(\Q)$  and $\Sel_{\rho_1}(\Q)$  respectively are surjective.
\end{enumerate}
Let $Q_{3,2}$ (and $Q_{2,1}$ respectively)  be a height two prime ideal in $R_{F,G}[[\Gamma]]$ (and $R_{F,G}$ respectively) containing $\ker(\pi_{3,2})$ (and $\ker(\pi_{2,1})$ respectively). To establish \ref{fin-proj}, it will suffice to prove that the projective dimensions of the $\left(R_{F,G}[[\Gamma]]\right)_{Q_{3,2}}$-modules
\begin{align*}
H^1\left(G_\Sigma, D_{\rho_3}\right)^\vee \otimes_{R_{F,G}[[\Gamma]]} R_{F,G}[[\Gamma]]_{Q_{3,2}}, \\ H^1\left(\Gal{\overline{\Q}_p}{\Q_p}, \frac{D_{\rho_3}}{\Fil^+ D_{\rho_3}}\right)^\vee \otimes_{R_{F,G}[[\Gamma]]} (R_{F,G}[[\Gamma]])_{Q_{3,2}},
\end{align*}
and the projective dimensions of the $\left(R_{F,G}\right)_{Q_{2,1}}$-modules
\begin{align*}
H^1\left(G_\Sigma, D_{\rho_2}\right)^\vee \otimes_{R_{F,G}} \left(R_{F,G}\right)_{Q_{2,1}}, \\ H^1\left(\Gal{\overline{\Q}_p}{\Q_p}, \frac{D_{\rho_2}}{\Fil^+ D_{\rho_2}}\right)^\vee \otimes_{R_{F,G}} (R_{F,G})_{Q_{2,1}},
\end{align*}
respectively are finite. The hypothesis \ref{van-h2} allows us to apply Proposition 5.5 in \cite{bharathwaj2016algebraic}. It will now suffice to prove the following statements \ref{reg-rho-3} and \ref{reg-rho-2}:

\begin{enumerate}[style=sameline, leftmargin=2cm, style=sameline, align=left,label=\textsc{Reg-$\rho_3$}, ref=\textsc{Reg-$\rho_3$}]
\item \label{reg-rho-3} For every height two prime ideal $Q_{3,2}$ containing $\ker(\pi_{3,2})$ and in the support of $H^0(G_{\Sigma},D_{\rho_3})^\vee$ or $H^0\left(\Gal{\overline{\Q}_p}{\Q_p},\frac{D_{\rho_3}}{\Fil^+D_{\rho_3}}\right)^\vee$, the 2-dimensional local ring $\left(R_{F,G}[[\Gamma]]\right)_{Q_{3,2}}$ is regular.
\end{enumerate}

\begin{enumerate}[style=sameline, leftmargin=2cm, style=sameline, align=left,label=\textsc{Reg-$\rho_2$}, ref=\textsc{Reg-$\rho_2$}]
\item \label{reg-rho-2} For every height two prime ideal $Q_{2,1}$ containing $\ker(\pi_{2,1})$ and in the support of $H^0(G_{\Sigma},D_{\rho_2})^\vee$,
\begin{itemize}
\item the $2$-dimensional local ring $\left(R_{F,G}\right)_{Q_{2,1}}$ is regular;
\item the $1$-dimensional ring $\left(\frac{R_{F,G}}{\ker(\pi_{2,1})}\right)_{Q_{2,1}}$ is integrally closed (and hence regular too).
\end{itemize}
The projective dimension of the $R_{F,G}$-module $H^0\left(\Gal{\overline{\Q}_p}{\Q_p}, \frac{D_{\rho_2}}{\Fil^+D_{\rho_2}}\right)^\vee$  equals  one.
\end{enumerate}
Proposition \ref{p-regular-prop}, Lemma \ref{p-invertible} and Lemma \ref{loc-tor-3-2} establish \ref{reg-rho-2}. Lemma 4.16 and Proposition 4.17 in \cite{bharathwaj2016algebraic} establish \ref{reg-rho-3}. \\

The control theorems that relate $\Sel_{\rho_3}(\Q)^\vee \otimes_{R_{F,G}[[\Gamma]]} R_{F,G}$ with $\Sel_{\rho_2}(\Q)^\vee$, given in equation (\ref{control-eqn-previous-3-2}), and $\Sel_{\rho_2}(\Q)^\vee \otimes_{R_{F,G}} R_{f,G}$ with $\Sel_{\rho_1}(\Q)^\vee$, given in equation (\ref{control-eqn-previous-2-1}), use Proposition 5.2 in \cite{bharathwaj2016algebraic}. The following flowchart summarizes the logical flow involved in those arguments: \\

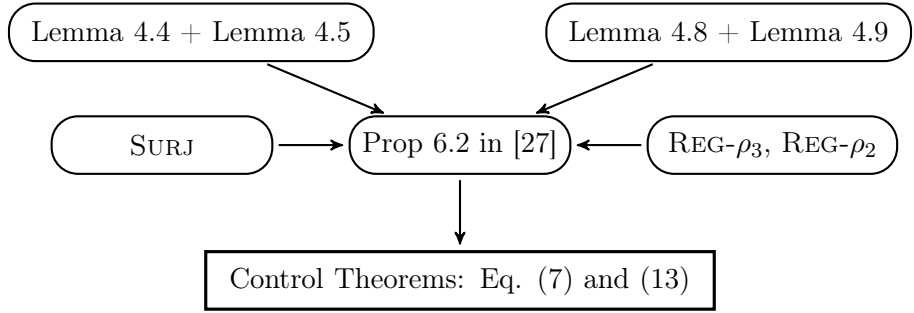
\begin{center}
\begin{tikzpicture}[node distance=1.5cm, auto]
\node[punkt, text width=17em,very thick, rounded corners=0pt](control-thms){Control Theorems:  Eq. (\ref{control-eqn-previous-3-2}) and (\ref{control-eqn-previous-2-1})};

\node[punkt, above = 1 cm of control-thms](prop-control){Prop 6.2 in \cite{bharathwaj2016algebraic}}edge[pil](control-thms);
\node[punkt, left = 1cm of prop-control]{\ref{surj}} edge[pil](prop-control);
\node[punkt, text width=8em, right = 1cm of prop-control]{\ref{reg-rho-3}, \ref{reg-rho-2}} edge[pil](prop-control);

\node[ above of = prop-control](dummy+control) {};
\node[punkt,text width=12 em, left  = 1 cm of dummy+control] {Lemma \ref{p-invertible} + Lemma \ref{global-tor-torsion}} edge[pil](prop-control);
\node[punkt,text width=12 em, right  = 1 cm of dummy+control] {Lemma \ref{loc-tor-3-2} + Lemma \ref{local-tor-torsion}} edge[pil](prop-control);

\end{tikzpicture}
\end{center}

\subsection*{Calculations involving global Galois cohomology groups}

\begin{lemma} \label{reg-loc-2}
Let $Q$ be a height two prime ideal in $R_{F,G}$ satisfying the following properties:
\begin{itemize}
\item $p \notin Q$.
\item $i_F(\P_{F}) \subset Q$, for some classical height one prime ideal $\P_{F}$ of $R_F$ with weight $ \geq 2$.
\end{itemize}
Then, $(R_{F,G})_Q$ is a regular local ring with Krull dimension $2$.
\end{lemma}

\begin{proof}
We have the following natural inclusions:
\begin{align*}
i_1: O[[x_F]] \hookrightarrow R_F, \qquad i_F : R_F \hookrightarrow R_{F,G}, \qquad i_{F,G} : R_G[[x_F]] \hookrightarrow R_{F,G}.
\end{align*}
Let us define the following prime ideals:
\begin{align*}
\underbrace{\p_0}_{\text{height}=1} := i_{1}^{-1}(\P_{F}),\qquad  \underbrace{Q_0}_{\text{height}=2} := i_{F,G}^{-1}(Q).
\end{align*}
Since $p \notin Q$ and the Krull dimension of $R_F$ equals two, we have the equality $i_F^{-1}(Q)=\P_{F}$. Furthermore, since $\P_{F}$ is a classical height one prime ideal in $R_{F}$, the extension $$(O[[x_F]])_{\p_0} \hookrightarrow (R_F)_{\P_{F}}$$ of discrete valuation rings is finite and \'etale (see Corollary 1.4 in \cite{hida1986galois}). Let $$\RRR = \left(R_{G}[[x_F]]\right)_{Q_0} \otimes_{(O[[x_F]])_{\p_0}} (R_F)_{\P_{F}}.$$ We have the following commutative diagrams:
\begin{align*}
\xymatrix{
O[[x_F]] \ar[r]\ar[d]^{i_1}& R_{G}[[x_F]] \ar[d]^{i_{F,G}}  \\
R_F  \ar[r]^{i_F} & R_{F,G}
} \qquad \xymatrix{
(O[[x_F]])_{\p_0} \ar[r]\ar[d]& \left(R_{G}[[x_F]]\right)_{Q_0} \ar[d]  \\
(R_F)_{\P_{F}}  \ar[r]& \RRR
}
\end{align*} \'Etale morphisms are stable under base change (Proposition 3.22 in Chapter 4 of \cite{MR1917232}). So, the morphism $\left(R_{G}[[x_F]]\right)_{Q_0} \rightarrow  \RRR$
is also \'etale. Since $\P_{F}$ is a classical height one prime ideal in $R_{F}$, there exists integers $n_1$ and $n_2$ such that $(1+x_F)^{n_1}-(1+p)^{n_2}$ belongs to $\P_{F}$ (and hence belongs to $Q_0$ as well). The polynomial $(1+x_F)^{n_1}-(1+p)^{n_2}$ is a non-constant monic polynomial. Lemma 1.16 in \cite{bharathwaj2016algebraic} now forces $(R_{G}[[x_F]])_{Q_0}$ to be a regular local ring with Krull dimension $2$. Since the morphism $\left(R_{G}[[x_F]]\right)_{Q_0} \rightarrow  \RRR$ is \'etale, the ring $\RRR$ is regular too (Corollary 3.24 in Chapter 4 of \cite{MR1917232}). \\

The natural $(O[[x_F]])_{\p_0}$-algebra maps
\begin{align*}
\left(R_{G}[[x_F]]\right)_{Q_0} \hookrightarrow \left(R_{F,G}\right)_{Q}, \qquad (R_F)_{\P_{F}} \hookrightarrow \left(R_{F,G}\right)_{Q}
\end{align*}
give us a natural $(O[[x_F]])_{\p_0}$-algebra map $\RRR \rightarrow \left(R_{F,G}\right)_{Q}$. If we let $Q_1$ be the prime ideal in $\RRR$ that $Q$ pullbacks to, we get a natural map $\beta: \RRR_{Q_1} \rightarrow \left(R_{F,G}\right)_Q$ of local rings.  Using the fact that the completed tensor product $R_{F,G}$ satisfies the universal property of being a co-product in the category of complete semi-local Noetherian $O$-algebras, we can deduce the following natural isomorphism: \begin{align}\label{universal-coproduct}
R_{G}[[x_F]] \otimes_{O[[x_F]]} R_F \cong R_{F,G}.
\end{align}
Equation (\ref{universal-coproduct})   gives us a natural map $R_{F,G} \rightarrow \RRR$. We obtain the two following sequence of maps:
\begin{align*}
R_{F,G} \rightarrow \RRR_{Q_1} \xrightarrow {\beta} (R_{F,G})_Q, \qquad \left(R_{F,G}\right)_Q \xrightarrow {\beta'} \RRR_{Q_1} \xrightarrow {\beta} \left(R_{F,G}\right)_Q
\end{align*}

The second sequence of maps in the equation above is obtained by localizing the first sequence of maps with respect to the prime ideal (corresponding to) $Q$. The composition $\beta \circ \beta'$ is the identity map. Thus, the map $\beta$ must be surjective. In fact, $\beta$ turns out to be an isomorphism. To see this, it suffices to show that $\ker(\beta)$ equals $0$. Since the Krull dimensions of both $\RRR_{Q_1}$ and $(R_{F,G})_Q$ equal two, the height of the prime ideal $\ker(\beta)$ must equal  $0$. Since $\RRR_{Q_1}$ is a domain (it is in fact a regular local ring), $\ker(\beta)$ must equal $0$.   This completes the proof since $\RRR_{Q_1}$ is a regular local ring.
\end{proof}

\begin{lemma} \label{no-supp-p}
The $R_{f,G}$-module $H^0(G_{\Sigma},D_{\rho_1})^\vee$ is not supported at any height $1$ prime ideal of $R_{f,G}$ containing the prime number $p$. Consequently, the $R_{f,G}$-module $H^0(G_{\Sigma},D_{\rho_1})^\vee$ is torsion.
\end{lemma}

\begin{proof}
Let $\q$ be a height one prime ideal in $R_{f,G}$ containining $p$. Let $K$ denote the fraction field of $\frac{R_{f,G}}{\q}$. Note that the characteristic of $K$ is equal to $p$. We have the following $G_\Sigma$-equivariant isomorphisms:
\begin{align*}
 & D_{\rho_1}^\vee \cong \Hom_{R_{f,G}}\left(L_G \otimes_{R_G} R_{f,G}\ ,\ L_f \otimes_{O_f} R_{f,G} \right),\\
\implies & D_{\rho_1}^\vee \otimes_{R_{f,G}} K  \cong \Hom_{K}\left(V_G,V_f\right).
\end{align*}
Here, $V_f$ and $V_G$ are $K$-vector spaces with a $G_\Sigma$-action given by following representations $\sigma_f$ and $\sigma_G$ respectively:
\begin{align*}
\sigma_f: G_\Sigma \xrightarrow {\rho_f} \Gl_2(O_f) \hookrightarrow \Gl_2(R_{f,G}) \rightarrow \Gl_2(K), \\
\sigma_G: G_\Sigma \xrightarrow {\rho_G} \Gl_2(R_G) \hookrightarrow \Gl_2(R_{f,G}) \rightarrow \Gl_2(K).
\end{align*}
To show that $H^0(G_\Sigma,D_{\rho_1})^\vee$ is not supported at $\p$, it will be sufficient to show that $H^0\left(G_\Sigma,\Hom_K\left(V_G,V_f\right)\right)$ is equal to $0$. Both $\sigma_f$ and $\sigma_G$ are Galois representations over $K$, a field of characteristic $p$. In fact as a $G_\Sigma$-representation, $\sigma_f$ is isomorphic to the residual representation $\overline{\rho}_F$, after extending scalars to $K$. The residual representation $\overline{\rho}_F$ is absolutely irreducible (due to the hypothesis \ref{residual-irr}). Consequently, the Galois representation $\sigma_f$ is irreducible. Thus, to show that $H^0\left(G_\Sigma,\Hom_K\left(V_G,V_f\right)\right)$ is equal to $0$, it will be sufficient to show that $\sigma_f$ and $\sigma_G$ are not isomorphic as $G_\Sigma$-representations over $K$. For this purpose, it will be sufficient to show that $\det(\sigma_f)$ and $\det(\sigma_G)$ are not isomorphic as $G_\Sigma$-representations. The image of $\det(\sigma_f)$ lies inside $\overline{\mathbb{F}}_p^\times$. The pullback $j_G^{-1}(\q)$ of the height one prime ideal $\q$ under the map  $j_G:R_G \hookrightarrow R_{f,G}$ is a height one prime ideal (and not the maximal ideal) of $R_G$. Consequently, the image of $\det(\sigma_G)$ lies inside $\overline{\mathbb{F}}_p[[x_G]]^\times$ but does not lie inside $\overline{\mathbb{F}}_p^\times$. Consequently, $\det(\sigma_f)$ and $\det(\sigma_G)$ are not isomorphic as $G_\Sigma$-representations. The lemma~follows.
\end{proof}

Note that $\frac{R_{F,G}}{\ker(\pi_{2,1})} \hookrightarrow R_{f,G}$ is an integral extension. We have the following isomorphism of $R_{f,G}$-modules:
\begin{align} \label{base-change-2-1}
\left(H^0(G_{\Sigma},D_{\rho_2})^\vee \otimes_{R_{F,G}} \frac{R_{F,G}}{\ker(\pi_{2,1})} \right)\otimes_{\frac{R_{F,G}}{\ker(\pi_{2,1})}} R_{f,G} \cong H^0(G_{\Sigma},D_{\rho_1})^\vee.
\end{align}

Let $Q_{2,1}$ be a height two prime ideal in $R_{F,G}$ containing $\ker(\pi_{2,1})$. It uniquely corresponds to a height one prime ideal $\p$ in $\frac{R_{F,G}}{\ker(\pi_{2,1})}$; And there exist finitely many height one prime ideals $\q_1,\ldots,\q_n$ in $R_{f,G}$ lying above $\p$.  A simple application of Nakayama's lemma along with equation (\ref{base-change-2-1}) lets us make the following observation: A height two prime ideal $Q_{2,1}$ in $R_{F,G}$ containing $\ker(\pi_{2,1})$ belongs to the support of $H^0(G_{\Sigma},D_{\rho_2})^\vee$ if and only if one of the height one prime ideals $\q_i$ (for $1 \leq i \leq n$) in $R_{f,G}$ belongs to the support $H^0(G_\Sigma,D_{\rho_1})^\vee$. These observations along with Lemma \ref{reg-loc-2} and Lemma \ref{no-supp-p} immediately give us the following proposition:

\begin{proposition} \label{p-regular-prop}
Let $Q_{2,1}$ be a height two prime ideal containing $\ker(\pi_{2,1})$.
\begin{itemize}
\item If $p \in Q_{2,1}$, then $\left(H^0(G_{\Sigma},D_{\rho_2})^\vee\right)_{Q_{2,1}}=0$.
\item If $p \notin Q_{2,1}$, then $\left(R_{F,G}\right)_{Q_{2,1}}$ is a regular local ring of dimension $2$.
\end{itemize}
\end{proposition}

Let $P_k$ be the classical height one prime ideal corresponding to $f$. Here, $P_k$ is defined to be the kernel of the natural map $\pi_f : R_F \rightarrow O_f$. The ring $O_f$, defined to be the integral closure of $\frac{R_F}{P_k}$, is the ring of integers in a finite extension of $\Q_p$. As a result, the index $\left[O_f:\frac{R_F}{\ker(\pi_f)}\right]$ is finite and equal to a power of $p$. The ring $R_{f,G}$ is the completed tensor product $O_f \hotimes R_G$ (over the ring $O$). The map $\pi_{2,1} : R_{F,G} \rightarrow R_{f,G}$ is obtained naturally via the maps
$$ \pi_f: R_F \rightarrow O_f, \qquad \text{id}: R_G \xrightarrow {=} R_G.$$ So, the cokernel of the natural inclusion  $\frac{R_{F,G}}{\ker(\pi_{2,1})} \hookrightarrow R_{f,G}$ is also annihilated by a power of $p$. We obtain the following lemma:

\begin{lemma} \label{p-invertible}
Let $\q$ be a height one prime ideal in $R_{f,G}$ not containing $p$. Let $Q_{2,1}$ denote the height two prime ideal in $R_{F,G}$, containing $\ker(\pi_{2,1})$, given by $\pi_{2,1}^{-1}(\q)$. Then, the following natural inclusion is an equality:
\begin{align*}
\left(\frac{\R_{F,G}}{\ker(\pi_{2,1})}\right)_{Q_{2,1}} \stackrel{=}{\hookrightarrow} \left(R_{f,G}\right)_\q.
\end{align*}
\end{lemma}

Lemma \ref{no-supp-p} also lets us conclude that $H^0(G_\Sigma,D_{\rho_1})^\vee$ is a torsion $R_{f,G}$-module. Combining equation (\ref{base-change-2-1}) with Nakayama's lemma,  we can then conclude that $\left(H^0(G_{\Sigma},D_{\rho_2})^\vee\right)_{\ker(\pi_{2,1})} = 0$. \\

What we also obtain is that $H^0(G_{\Sigma},D_{\rho_2})^\vee$ is a torsion $R_{F,G}$-module. Considering $R_{F,G}$ as an $R_{F,G}[[\Gamma]]$-module via the map $\pi_{3,2}:R_{F,G}[[\Gamma]] \rightarrow R_{F,G}$, we also have the following isomorphism of $R_{F,G}$-modules:
\begin{align} \label{base-change-3-2}
H^0(G_{\Sigma},D_{\rho_3})^\vee \otimes_{R_{F,G}[[\Gamma]]} R_{F,G} \cong H^0(G_{\Sigma},D_{\rho_2})^\vee.
\end{align}
Nakayama's lemma lets us conclude that $\left(H^0(G_{\Sigma},D_{\rho_3})^\vee\right)_{\ker(\pi_{3,2})} = 0$. We have shown that the following equality holds:
\begin{align*}
\left(H^0(G_{\Sigma},D_{\rho_3})^\vee\right)_{\ker(\pi_{3,2})} = \left(H^0(G_{\Sigma},D_{\rho_2})^\vee\right)_{\ker(\pi_{2,1})} = 0.
\end{align*}
Since the Tor functor commutes with localization,  we immediately get the following lemma:
\begin{lemma}\label{global-tor-torsion}
For all $i \geq 0$, the following statements hold:
\begin{itemize}
\item The $R_{F,G}$-module $\Tor_i^{R_{F,G}[[\Gamma]]}\left( R_{F,G}, \ H^0(G_\Sigma,D_{\rho_3})^\vee \right)$ is torsion,
\item The $R_{f,G}$-module $\Tor_i^{R_{F,G}}\left( R_{f,G}, \ H^0(G_\Sigma,D_{\rho_2})^\vee \right)$ is torsion.
\end{itemize}
\end{lemma}

\subsection*{Calculations involving local Galois cohomology groups}

As for the local Galois cohomology groups, we will begin with the following exact sequence of $R_{F,G}[[\Gamma]]$-modules that is $\Gal{\overline{\Q}_p}{\Q_p}$-equivariant:
\begin{align}
0 & \rightarrow \underbrace{W}_{\text{corank}=1}  \rightarrow \underbrace{\frac{D_{\rho_3}}{\Fil^+D_{\rho_3}}}_{\text{corank}=2} \rightarrow  \underbrace{W'}_{\text{corank}=1} \rightarrow 0, \text{ where}
\end{align}

\begin{align*}
W := \Hom_{R_{F,G}}\left(\frac{L_{F}}{\Fil^+L_F} \otimes_{R_F} R_{F,G},\ \frac{L_G}{\Fil^+L_G} \otimes_{R_G} R_{F,G}\right) \otimes_{R_{F,G}} \hat{R_{F,G}[[\Gamma]]} (\kappa^{-1}), \\  W':= \Hom_{R_{F,G}}\left(\Fil^+L_F \otimes_{R_F} R_{F,G},\ \frac{L_G}{\Fil^+L_G} \otimes_{R_G} R_{F,G}\right) \otimes_{R_{F,G}} \hat{R_{F,G}[[\Gamma]]}(\kappa^{-1}).
\end{align*}
Note that we also have the following $\Gal{\overline{\Q}_p}{\Q_p}$-equivariant isomorphisms of $R_{F,G}[[\Gamma]]$-modules:
\begin{align*}
&W \cong R_{F,G}\bigg((i_F \circ \epsilon_F^{-1}) \cdot (i_G \circ \epsilon_G)\bigg) \otimes_{R_{F,G}} \hat{R_{F,G}[[\Gamma]]}(\kappa^{-1}), \\& W' \cong R_{F,G}\bigg((i_F \circ \delta_F^{-1}) \cdot (i_G \circ \epsilon_G)\bigg) \otimes_{R_{F,G}} \hat{R_{F,G}[[\Gamma]]}(\kappa^{-1}).
\end{align*}
The hypotheses \ref{p-distinguished-inertia} ensures us that the residual representation, defined over $\overline{\mathbb{F}}_p$, associated to $(i_F \circ \delta_F^{-1}) \cdot (i_G \circ \epsilon_G)$ is non-trivial. Consequently, $H^0(I_p,W')=0$. So, we get the following $\frac{\Gal{\overline{\Q}_p}{\Q_p}}{I_p}$-equivariant isomorphism of $R_{F,G}[[\Gamma]]$-modules:
\begin{align}
H^0\left(I_p,\frac{D_{\rho_3}}{\Fil^+D_{\rho_3}}\right) \cong H^0(I_p,W).
\end{align}
Let $\eta_p$ denote the unique prime in $\Q_\infty$ lying above $p$ and let $I_{\eta_p}$ denote the corresponding inertia subgroup at $\eta_p$. The Galois representation $\rho_3$ is related to the cyclotomic deformation of $\rho_2$. We have denoted the quotient $\frac{\Gal{\overline{\Q}_p}{\Q_p}}{I_p}$ by $\Gamma_p$. The arguments in Section 3 of \cite{greenberg1994iwasawa} combined with the fact that both the characters $\epsilon_F$ and $\epsilon_G$ are unramified at $p$ give us the following isomorphism  that is $\Gamma_p$-equivariant:
\begin{align*}
H^0(I_p,W) \cong H^0\left(I_{\eta_p},\hat{R_{F,G}}\bigg((i_F \circ \epsilon_F^{-1}) \cdot (i_G \circ \epsilon_G)\bigg)\right) \cong \hat{R_{F,G}} \bigg((i_F \circ \epsilon_F^{-1}) \cdot (i_G \circ \epsilon_G)\bigg)
\end{align*}
Taking the Pontryagin duals of the modules in the above equation gives us the following lemma:
\begin{lemma} \label{3-inertia-free-iso} We have the following isomorphism of $R_{F,G}[[\Gamma]]$-modules  that is $\Gamma_p$-equivariant.
$$H^0\left(I_p,\frac{D_{\rho_3}}{\Fil^+D_{\rho_3}}\right)^\vee = R_{F,G} \bigg((i_F \circ \epsilon_F) \cdot (i_G \circ \epsilon_G^{-1})\bigg).$$
Here, we consider $R_{F,G}$ as an $R_{F,G}[[\Gamma]]$-module via the map $\pi_{3,2} : R_{F,G}[[\Gamma]] \rightarrow R_{F,G}$.
\end{lemma}
Note that $a_p$ and $b_p$ are elements of $R_F$ and $R_G$ respectively but not elements of $O$. This is because the values that $a_p$ and $b_p$ take at classical specializations of $F$ and $G$ respectively vary as one varies the weight (see Lemma 3.2 in \cite{hida1985ap}). Consequently, the ratio $\frac{i_F(a_p)}{i_G(b_p)}$ is a not an element of $O$ (in particular, it is not equal to $1$). Thus, the following map of free $R_{F,G}$-modules (of rank $1$), which is given by multiplication by $\frac{i_F(a_p)}{i_G(b_p)}-1$, is  injective:
\begin{align*}
H^0\left(I_p,\frac{D_{\rho_3}}{\Fil^+D_{\rho_3}}\right)^\vee \xrightarrow {\Frob_p-1} H^0\left(I_p,\frac{D_{\rho_3}}{\Fil^+D_{\rho_3}}\right)^\vee.
\end{align*}
The Frobenius at $p$, denoted by $\Frob_p$, is a topological generated for $\Gamma_p$. This gives us the following equality:
$$H^1\left(\Gamma_p,H^0\left(I_p,\frac{D_{\rho_3}}{\Fil^+D_{\rho_3}}\right)\right)=0.$$

Arguing similarly, we get the following results for the modules related to $\rho_2$ and $\rho_1$:
 \begin{align} \label{2-1-module-structures}
 H^0\left(I_p,\frac{D_{\rho_2}}{\Fil^+D_{\rho_2}}\right) &\cong \hat{R_{F,G}}\left((i_F \circ \epsilon_F^{-1})\cdot(i_G \circ \epsilon_G)\right),  \\ \notag H^1\left(\Gamma_p,\ H^0\left(I_p,\frac{D_{\rho_2}}{\Fil^+D_{\rho_2}}\right)\right)&=0. \\ \notag
 H^0\left(I_p,\frac{D_{\rho_1}}{\Fil^+D_{\rho_1}}\right) &\cong \hat{R_{f,G}}\left((j_F \circ \epsilon_F^{-1})\cdot(i_G \circ \epsilon_G)\right),  \\ \notag H^1\left(\Gamma_p,\ H^0\left(I_p,\frac{D_{\rho_1}}{\Fil^+D_{\rho_1}}\right)\right)&=0.
\end{align}

Combining these observations, we get the following lemma:
\begin{lemma} \label{strict-selmer-difference} Let $i \in \{1,2,3\}$. Then,
$H^1\left(\Gamma_p,\ H^0\left(I_p,\frac{D_{\rho_i}}{\text{Fil}^+D_{\rho_i}}\right)\right)=0$.
\end{lemma}

Consider the short exact sequence $0 \rightarrow R_{F,G}[[\Gamma]] \xrightarrow {\gamma_0-1} R_{F,G}[[\Gamma]] \rightarrow R_{F,G} \rightarrow 0$ of $R_{F,G}[[\Gamma]]$-modules. Here, $\gamma_0$ is a topological generator for the pro-cyclic group $\Gamma$. Taking the tensor product with the $R_{F,G}[[\Gamma]]$-module $H^0\left(I_p,\frac{D_{\rho_3}}{\Fil^+D_{\rho_3}}\right)^\vee$ over the ring $R_{F,G}[[\Gamma]]$, we get the following exact sequence:

\begin{align*}
0 \rightarrow& \Tor_1^{R_{F,G}[[\Gamma]]} \left( R_{F,G}, \ H^0\left(I_p,\frac{D_{\rho_3}}{\Fil^+D_{\rho_3}}\right)^\vee \right) \rightarrow \\ &\rightarrow H^0\left(I_p,\frac{D_{\rho_3}}{\Fil^+D_{\rho_3}}\right)^\vee  \xrightarrow {\gamma_0-1} H^0\left(I_p,\frac{D_{\rho_3}}{\Fil^+D_{\rho_3}}\right)^\vee \rightarrow \\ & \rightarrow H^0\left(I_p,\frac{D_{\rho_3}}{\Fil^+D_{\rho_3}}\right)^\vee \otimes_{R_{F,G}[[\Gamma]]} R_{F,G}  \rightarrow 0.
\end{align*}

By Nakayama's lemma, the natural surjection $$H^0\left(I_p,\frac{D_{\rho_3}}{\Fil^+D_{\rho_3}}\right)^\vee  \rightarrow H^0\left(I_p,\frac{D_{\rho_3}}{\Fil^+D_{\rho_3}}\right)^\vee \otimes_{R_{F,G}[[\Gamma]]} R_{F,G}$$ must be an isomorphism since it is a map between free $R_{F,G}$-modules of rank $1$ (see Lemma \ref{3-inertia-free-iso} and equation (\ref{2-1-module-structures})). This tells us that the map $H^0\left(I_p,\frac{D_{\rho_3}}{\Fil^+D_{\rho_3}}\right)^\vee  \xrightarrow {\gamma_0-1} H^0\left(I_p,\frac{D_{\rho_3}}{\Fil^+D_{\rho_3}}\right)^\vee$, given by multiplication by the element $\gamma_0-1$, is the zero map. Using Lemma \ref{3-inertia-free-iso}, we get the following isomorphism of $\R_{F,G}$-modules that is $\Gamma_p$-equivariant:
 \begin{align*}
\Tor_1^{R_{F,G}[[\Gamma]]} \left( R_{F,G}, \ H^0\left(I_p,\frac{D_{\rho_3}}{\Fil^+D_{\rho_3}}\right)^\vee \right) & \cong H^0\left(I_p,\frac{D_{\rho_3}}{\Fil^+D_{\rho_3}}\right)^\vee  \\ &\cong R_{F,G} \bigg( (i_F \circ \epsilon_F)\cdot (i_G \circ \epsilon_G^{-1})\bigg).
\end{align*}
This observation, along with equation (\ref{2-1-module-structures}), immediately gives us the following lemma:
\begin{lemma} \label{loc-tor-3-2}
We have the following isomorphism of torsion $R_{F,G}$-modules:
\begin{align*}
    \Tor_1^{R_{F,G}[[\Gamma]]} \left( R_{F,G}, \ H^0\left(I_p,\frac{D_{\rho_3}}{\Fil^+D_{\rho_3}}\right)^\vee \right)_{\Gamma_p} & \cong \frac{R_{F,G}}{\left(1-\frac{i_F(a_p)}{i_G(b_p)}\right)}, \\ H^0\left(\Gal{\overline{\Q}_p}{\Q_p},\frac{D_{\rho_2}}{\Fil^+D_{\rho_2}}\right)^\vee & \cong \frac{R_{F,G}}{\left(1-\frac{i_F(a_p)}{i_G(b_p)}\right)}.
\end{align*}
\end{lemma}

\begin{lemma}\label{local-tor-torsion}
\mbox{}
\begin{itemize}
\item For all $i \geq 2$, we have  $\Tor_i^{R_{F,G}[[\Gamma]]}\left(R_{F,G},\ H^0\left(I_p,\frac{D_{\rho_3}}{\Fil^+D_{\rho_3}}\right)^\vee\right)=0$.
\item For all $i \geq 1$, we have $\Tor_i^{R_{F,G}} \left( R_{f,G}, \ H^0\left(I_p,\frac{D_{\rho_2}}{\Fil^+D_{\rho_2}}\right)^\vee \right)=0$.
\end{itemize}
\end{lemma}

\begin{proof}
The first statement simply follows from the fact that as an $R_{F,G}[[\Gamma]]$-module, the projective dimension of $R_{F,G}$ (which is isomorphic to $\frac{R_{F,G}[[\Gamma]]}{(\gamma_0-1)}$ as an $R_{F,G}[[\Gamma]]$-module) is equal to $1$. The second statement follows from equation (\ref{2-1-module-structures}), since $H^0\left(I_p,\frac{D_{\rho_2}}{\Fil^+D_{\rho_2}}\right)^\vee $ is a free $R_{F,G}$-module.
\end{proof}

\begin{lemma} \label{van-h2-rho-2-1}
\begin{align*}
H^2\left(\Gal{\overline{\Q}_p}{\Q_p},D_{\rho_1}\right) &= H^2\left(\Gal{\overline{\Q}_p}{\Q_p},\Fil^+D_{\rho_1}\right) \\ & = H^2\left(\Gal{\overline{\Q}_p}{\Q_p},D_{\rho_2}\right) = H^2\left(\Gal{\overline{\Q}_p}{\Q_p},\Fil^+D_{\rho_2}\right) = 0.
\end{align*}
\end{lemma}

\begin{proof}
We will show that $H^2\left(\Gal{\overline{\Q}_p}{\Q_p},D_{\rho_2}\right) =0$. The rest of the lemma follows using similar arguments. By local duality, we have
\begin{align*}
H^2\left(\Gal{\overline{\Q}_p}{\Q_p},D_{\rho_2}\right) &\cong H^0\left(\Gal{\overline{\Q}_p}{\Q_p},L^*_{2}\right), \\ \notag & \text{where } L_2^* \cong \Hom_{R_{F,G}}\left(L_G \otimes_{R_G} R_{F,G}, \ L_F \otimes_{R_F} R_{F,G} (\chi_p)\right).
\end{align*}
Here, $\chi_p : G_\Sigma \rightarrow \Z_p^\times$ is the $p$-adic cyclotomic character given by the action of $G_\Sigma$ on $\mu_{p^\infty}$, the $p$-power roots of unity. Let $E$ denote the fraction field of $R_{F,G}$. We define the following vector spaces over $E$ that have an action of $\Gal{\overline{\Q}_p}{\Q_p}$:
\begin{align*}
V_F(\chi_p) := L_F \otimes_{R_F} E(\chi_p), \qquad V_G := L_G \otimes_{R_G} E.
\end{align*}
It will be sufficient to show that
\begin{align} \label{inertia-invariant}
\Hom_E\left(V_G,V_F(\chi_p)\right)^{I_p}\stackrel{?}{=} 0.
\end{align}
Both $V_F(\chi_p)$ and $V_G$ are reducible representations for the action of $\Gal{\overline{\Q}_p}{\Q_p}$. Let the $V_F(\chi_p)^{ss}$ and $V_G^{ss}$ denote the semi-simplifications of $V_F(\chi_p)$ and $V_G$ respectively for the action of the inertia subgroup  $I_p$. We have the following $I_p$-equivariant isomorphisms:
\begin{align*}
V_F(\chi_p)^{ss} \cong E\bigg(\chi_p (i_F \circ \delta_F) \bigg)\oplus E(\chi_p), \qquad V_G^{ss} \cong E(i_G \circ \delta_G) \oplus E.
\end{align*}
In this paragraph, we will consider characters over the inertia group $I_p$. The values of the character $\chi_p (i_F \circ \delta_F)$ lie inside the ring $R_F$ but not inside the ring $O$. Similarly, the values of the character $i_G \circ \delta_G$ lie inside the ring $R_G$ but not inside the ring $O$. As a result, the character $\chi_p (i_F \circ \delta_F)$ cannot equal $i_G \circ \delta_G$ or the trivial character. Also, the $p$-adic cyclotomic character $\chi_p$ (which takes values in $\Z_p^\times$) cannot equal the character $i_G \circ \delta_G$ or the trivial character. Equation (\ref{inertia-invariant}) now follows from these observations; and hence so does the lemma.
\end{proof}

The validity of the Weak Leopoldt conjecture \ref{weak-leo} along with Proposition 6.1 in \cite{MR2290593} and Proposition 5.2.4 in \cite{greenberg2010surjectivity} gives us the following corollary to Lemma \ref{van-h2-rho-2-1} .
\begin{corollary} \label{global-h2-vanish}
$H^2\left(G_\Sigma,D_{\rho_1}\right)=H^2\left(G_\Sigma,D_{\rho_2}\right)= H^2\left(G_\Sigma,D_{\rho_3}\right)=0.$
\end{corollary}

\section{Examples when $R_F$ is not regular} \label{examples}

To end, we will provide some examples when $R_F$, the normalization of a primitive component of the ordinary Hecke algebra, is not a UFD (and hence not regular). These examples will be based on the circle of ideas developed in \cite{hida1998global}, \cite{cho1999deformation} and \cite{cho2003deformations}. We have used Sage \cite{sage} to perform the computations.  Let $E=\Q(\sqrt{D})$ be a real quadratic field with discriminant $D > 0$, and where the prime $p$ splits; we have $(p) = \p_1 \p_2$ for two distinct prime ideals $\p_1$ and $\p_2$ in the ring of integers of the number field $E$. Let $k$ be an even integer strictly greater than $2$. Let $\chi_E : \Gal{\overline{\Q}}{\Q} \rightarrow \{\pm 1\}$ be the quadratic character associated to  $E$. Let $\epsilon$ be a fundamental unit for $E$. Let us suppose that the following condition holds:
\begin{align}\label{divisibility-cusp}
p \text{ divides } \text{Norm}_{E/\Q}\left(\epsilon^{k-1}-1\right).
\end{align}
We will assume as a result, without loss of generality, that $\epsilon^{k-1}-1 \in \mathfrak{p}_1$. Equation (\ref{divisibility-cusp}) allows us to apply Proposition 2.1 in \cite{hida1998global}. We have an ordinary $p$-stabilized newform $f = \sum \limits_{n \geq 1} a_n(f) q^n$ in $S_k^{\text{ord}}\left(Dp,\chi_E\right)$ such that the residual representation $\overline{\rho}_f$ (defined over a finite field $\mathbb{F}$ with characteristic $p$) associated to $f$ satisfies the following properties:
\begin{align*}
\overline{\rho}_f \text{ is absolutely irreducible}, \quad \overline{\rho}_f \otimes \chi_E \cong \overline{\rho}_f, \quad \text{$\overline{\rho}_f$ is $p$-ordinary (in the sense of  \cite{MR1854117})}.
\end{align*}
Let $O$ denote the ring of integers in a finite extension of $\Q_p$ containing the unique unramified quadratic extension of $\Q_p$, the ring of Witt vectors $W(\mathbb{F})$ and the $p^r$-th roots of unity (where $p^r$ denotes the exponent of the $p$-Sylow subgroup of the ideal class group of $E$).  Let $h^{\text{ord}}(Dp^\infty,\chi_{E},O)$ denote the universal $p$-ordinary Hecke algebra, with coefficients in $O$, with tame level $D$ and whose prime-to-$p$ part of the Nebentypus equals the character $\chi_E$. Let $h_\m$ be the unique local ring of $h^{\text{ord}}(Dp^\infty,\chi_{E},O)$ through which the following ring homomorphism factors:
\begin{align*}
& \lambda_f : h^{\text{ord}}(Dp^\infty,\chi_{E},O) \rightarrow \overline{\mathbb{F}}_p, \quad \text{ given by} \\
&\lambda_f(T_l) = \overline{a_l(f)}, \forall \text{ primes } l \text{ not dividing } Dp, \\ & \lambda_f(U_l) = \overline{a_l(f)},\  \forall \ \text{primes} \ l \text{ dividing } Dp.
\end{align*}
See \cite{hida1986galois} for a definition of the Hecke operators $T_l$ and $U_l$. Also note that $h_\m$ is the ``full'' Hecke algebra generated by the Hecke operators $T_l$ along with the diamond operators $<l>$ for primes $l$ not dividing $Dp$ and the $U_l$ operators for primes $l$ dividing $Dp$. Let $\rho_\m : G_\Q \rightarrow \Gl_2(h_\m)$ denote the Galois representation (section 2.4 of \cite{emerton2006variation}) such that  $\text{Trace}(\rho_\m(\Frob_l))=T_l$, for all primes $l \nmid Dp$. \\

Let $\text{CNL}_O$ denote the category of complete, Noetherian, local $O$-algebras whose residue field equals the residue field of $O$. Let $\text{SETS}$ denote the category of sets. Consider the deformation functor $\mathfrak{F} : \text{CNL}_O \rightarrow \text{SETS}$ that, for every $B \in \text{CNL}_O$, is given by strict equivalence classes of deformations $\varrho : G_\Q \rightarrow \Gl_2(B)$ that are of type $\mathfrak{D}$ and such that the residual representation associated to the deformation $\varrho$ is isomorphic to $\overline{\rho}_f$. A Galois representation $\varrho : G_\Q \rightarrow \Gl_2(B)$ is said to be of type $\mathfrak{D}$ if it satisfies the following properties:
\begin{itemize}
\item $\varrho$ is unramified at all primes $l$ not dividing $Dp$.
\item The restriction $\varrho \mid_{\Gal{\overline{\Q}_p}{\Q_p}}$ to the decomposition group at $p$  is $p$-ordinary.
\item If $l$ is a prime dividing $D$, we have $\rho \mid_{\Gal{\overline{\Q}_l}{\Q_l}} \ \sim  \ \left(\begin{array}{cc} 1 & 0 \\ 0 & \chi_E \end{array}\right)$.
\end{itemize}
Suppose also that the following condition holds:
\begin{align}
2(k-1) \not \equiv 0 \ (\text{mod } p-1).
\end{align}
By Theorem 3.9 in \cite{MR1854117}, the ring $h_\m$ along with the Galois representation $\rho_\m : G_\Q \rightarrow \GL_2(h_\m)$  represents the deformation functor $\mathfrak{F}$. The isomorphism $\overline{\rho}_f \cong \overline{\rho}_f \otimes \chi_E$ induces an involution $\tau : h_\m \rightarrow h_\m$ that is $O[[x_F]]$-equivariant. Let $h_\m^+$ denote the subring of $h_\m$ fixed by the involution $\tau$. We have the following natural inclusions of $O[[x_F]]$-algebras:
\begin{align*}
\O[[x_F]] \longhookrightarrow h_\m^+ \longhookrightarrow h_\m.
\end{align*}

Theorem 1.1 and 1.2 in \cite{hida1986galois} give us a natural surjection $h_\m \twoheadrightarrow h_{k,\m}$. Here, the ring $h_{k,\m}$ denotes the local factor of $h_k^{\text{ord}}(\Gamma_1(Dp),\chi_E, O)$  (the ``full'' Hecke algebra acting on the space of ordinary cuspforms of weight $k$ with Nebentypus $\chi_E$ and conductor $Dp$ generated by the operators $
T^{(k)}_l$ for primes $l \nmid Dp$ and $U^{(k)}_l$ for primes $l \mid Dp$) through which the following ring homomorphism factors:
\begin{align*}
& \beta_f: h_k^{\text{ord}}(\Gamma_1(Dp),\chi_E,O) \rightarrow \overline{\mathbb{F}}_p, \\ & \beta_f\left(T^{(k)}_l\right) = \overline{a_l(f)}, \ \forall l  \nmid Dp, \quad  \beta_f\left(U^{(k)}_l\right) = \overline{a_l(f)}, \ \forall l \mid Dp.
\end{align*}
Let us now suppose that the following condition holds:
\begin{align} \label{len-hecke-algebra}
\text{Rank}_{O}   \ h_{k,\m}=2.
\end{align}
Nakayama's lemma and the fact $h_\m$ is a  free $O[[x_F]]$-algebra (Theorem 3.1 in \cite{MR868300}) then let us conclude  that $\text{Rank}_{O[[x_F]]}   \ h_{\m}=2$. Note that by Corollary 3.12 in \cite{cho2003deformations}, the natural inclusion $h^+_\m \stackrel{\not \cong}{\hookrightarrow} h_\m$ is not an isomorphism. So, the natural inclusion $O[[x_F]] \stackrel{\cong}{\hookrightarrow} h^+_\m$ must be an isomorphism. By Theorem 1.1 in \cite{cho1999deformation}, we have the following isomorphism of $O[[x_F]]$-algebras:
\begin{align}
h_\m \cong  h^+_m \left[\vartheta\right] \cong O[[x_F]] [\vartheta], \quad
\text{where $\vartheta^2 = \sqrt{\left((1+x_F)^{\frac{\log(\epsilon)}{\log(1+p)}}-1\right)}$}.
\end{align}

The function $\log$ denotes the $p$-adic logarithm on $E$ that is determined by the inclusion $E \hookrightarrow E_{\p_1}$, where the field $E_{\p_1}$ denotes the completion of the field $E$ with respect to the valuation $\val_{\p_1}$ given by the prime ideal $\p_1$. Note that $E_{\p_1}\cong \Q_p$. Now, let us further suppose that $\epsilon$ is a $p$-th power in $E_{\p_1}$.  This condition is satisfied precisely when the following inequality holds:
\begin{align}\label{pth-power}
\val_{\p_1}\left(\epsilon^{p-1}-1\right) \geq 2.
\end{align}
In this case, the element $(1+x_F)^{\frac{\log(\epsilon)}{\log(1+p)}}-1$ is a square-free polynomial in $O[[x_F]]$ and a product of at least $2$ distinct irreducible elements of $O[[x_F]]$. As a result,  the local ring $h_\m$ is an integrally closed domain that is not a UFD. Finally, we let $F$ equal the unique Hida family passing through~$f$.

\begin{Remark}
The Galois representation $\rho_\m$, as constructed in Section 2.4 of \cite{emerton2006variation}, takes values in the ``reduced'' ordinary Hecke algebra $h_\m^{\text{red}}$ (generated by the Hecke operators $T_l$ and the diamond operators $<l>$, for primes $l$ not dividing $Dp$). Note that $h_\m^{\text{red}}$ is a subring of $h_\m$. Theorem 3.9 in \cite{MR1854117}, in fact, identifies the universal deformation ring for the functor $\mathfrak{F}$ with $h_\m^{\text{red}}$. It is Proposition 2.4.2 in \cite{emerton2006variation} that provides us an isomorphism $h_\m^{\text{red}} \cong h_\m$. See Theorem 12 in \cite{MR1039770} that also addresses this fact.
 \end{Remark}

\begin{Remark}\label{sagecode}
We will now explain how to ascertain that $\text{Rank}_O \ h_{k,\m}=2$. Let us define the following set of Hecke eigenforms of conductor $Dp$, weight $k$ and Nebentypus $\chi_E$ with coefficients in $\overline{\Q}_p$:
{
\begin{align*}
\mathcal{B}_{k,\m,st} = \{g \ \big|  & \quad g \text{ is a $p$-stabilized Hecke eigenform in $S^{\text{ord}}_k(Dp,\chi_E)$},  \\ & \quad a_1(g)=1, \ a_l(f) \equiv a_l(g) \ (\text{mod } l), \ \forall l \nmid Dp \}.
\end{align*}
}
By Corollary $3.7$ in \cite{hida1986galois}, it suffices to show that $|\mathcal{B}_{k,\m,st}|=2$. Note that since $f$ and its twist $f \otimes \chi_E$ belong to the set $\mathcal{B}_{k,\m,st}$, we have $|\mathcal{B}_{k,\m,st}|\geq2$. For each $g \in \mathcal{B}_{k,\m,st}$, let $g_0$ be the unique newform associated to it (Theorem 3.22 in \cite{hida2000modular}). Throughout this remark, it will be useful to keep in mind that since the weight $k$ is strictly greater than $2$ and that the conductor of $\chi_E$ is equal to $D$ (and hence of order prime to $p$), the conductor of $g_0$ is also equal to $D$ (Proposition 3.1 in \cite{hida1985ap}). We will now use the following code in Sage to construct a set $\mathcal{B}_{full}$, whose $\Gal{E}{\Q}$-orbit equals the set of all normalized (new) Hecke eigenforms of weight $k$, conductor $D$ and Nebentypus~$\chi_E$:
\begin{sageblock}
%sage    B_full = Newforms(chi_E, k, names='alpha');
\end{sageblock}
 For each $g \in \mathcal{B}_{full}$, let $K_g$ denote the number field generated by the Hecke eigenvalues of $g$. Note that if the $p$-stabilization of $g$ belongs to $\mathcal{B}_{k,\m,st}$ and if $l$ is a prime  that remains inert in the extension $E/\Q$, then $l$ divides $\text{Norm}^{K_g}_\Q(a_l(g))$. This is because the residual Galois representation $\overline{\rho}_g$  associated to $g$ (which is isomorphic to $\overline{\rho}_f$) would be isomorphic to  $\overline{\rho}_g \otimes \chi_E$ (since $\overline{\rho}_f \cong  \overline{\rho}_f \otimes \chi_E$). Define a subset $\mathcal{B}'_{k,\m,new}$ of $\mathcal{B}_{full}$ which contains those newforms $g$ such that $l$ divides $\text{Norm}^{K_g}_\Q(a_l(g))$, for all primes $l$ that remain inert in the extension $E/\Q$ and such that $3 \leq l \leq 100$. In the examples listed in Table \ref{table_examples}, $|\mathcal{B}'_{k,\m,new}|$ turns out to be $1$. And hence in those examples, the $\Gal{E}{\Q}$-orbit of $\mathcal{B}'_{k,\m,new}$ is in bijection with the set $\mathcal{B}_{k,\m,st}$ (the bijection is obtained through the process of $p$-stabilization) leading us to conclude that $|\mathcal{B}_{k,\m,st}| = 2 \times |\mathcal{B}'_{k,\m,new}|=2$. This will ascertain that $\text{Rank}_O \ h_{k,\m}=2$.
\end{Remark}

\begin{Remark}
Fixing a quadratic field $\Q(\sqrt{D})$ and letting the primes $p$ vary, it seems computationally difficult to find a large number of primes $p$ such that equation (\ref{pth-power}) holds.  William Stein informed us that the complexity of the code in Remark \ref{sagecode} is $\mathcal{O}\left((Dk)^3\right)$. This explains the paucity of our~examples.
\end{Remark}

\setlength\extrarowheight{18pt}

\begin{minipage}{\linewidth}

\centering
\captionof{table}{$R_F = h_\m = O[[x_F]]\left[\sqrt{\left(1+x_F\right)^{\frac{\log(\epsilon)}{\log(1+p)}}-1}\right]$} \label{table_examples}
\begin{tabularx}{\textwidth}{|p{2cm}|p{1.5cm}| p{2cm}|X|p{2.5cm}| p{1.5cm}|}
\hline
 $\Q(\sqrt{D})$ & $(k, p)$ & $\epsilon$ &  $\text{Norm}_{E/\Q}\left(\epsilon^{k-1}-1\right)$ & $\val_{\p_1}\left(\epsilon^{p-1}-1\right)$ & $[K_f:\Q]$ \\ \hline

$\Q(\sqrt{10 * 4})$ &  $(20, 191)$ & $\sqrt{10}+3$  & $-1 * 2 * 3 * 191^2 * 4523 * 1021973$
&$2$ & $112$  \\ \hline

$\Q(\sqrt{33})$ &  $(10, 37)$ & $4\sqrt{33} + 23$  & $-1 * 2^2 * 11 * 37^4 * 47^2 * 71^2$
&$2$ & $32$ \\ \hline
$\Q(\sqrt{89})$ &  $(4, 5)$ & $53\sqrt{89} + 500$  & $-1 * 2^3 * 5^3 * 1000003$
&$3$ & $22$ \\ \hline
$\Q(\sqrt{89})$ &  $(6, 5)$ & $53\sqrt{89} + 500$  & $-1 * 2^3 * 5^4 * 11 * 18181909091
$
&$3$ & $36$ \\ \hline
$\Q(\sqrt{629})$ &  $(4, 5)$ & $\frac{\sqrt{629}+25}{2}$ & $-1*2^2 * 5^2 * 157$ &$2$ & $168$ \\ \hline
\end{tabularx}

\end{minipage}

\subsubsection*{Acknowledgements} We are greatly indebted to Ralph Greenberg. This project was completed while the author was his student at University of Washington. We would like to thank Haruzo Hida for indicating to us how to compute the examples in Section \ref{examples}. We are grateful to Antonio Lei, Jeanine Van Order and the referee for their helpful comments.

\bibliography{biblio}
\bibliographystyle{plain}

\end{document}